\documentclass[12pt]{amsart}

\usepackage[text={425pt,650pt},centering]{geometry}
\usepackage{color}
\usepackage{esint,amssymb}
\usepackage{graphicx}
\usepackage{tikz}

\newtheorem{theorem}{Theorem}
\newtheorem{proposition}[theorem]{Proposition}
\newtheorem{lemma}[theorem]{Lemma}
\newtheorem{corollary}[theorem]{Corollary}

\theoremstyle{definition}
\newtheorem{remark}[theorem]{Remark}
\newtheorem{example}[theorem]{Example}
\newtheorem{definition}[theorem]{Definition}

\newtheorem{question}[theorem]{Question}

\newcommand{\tref}[1]{Theorem~\ref{t.#1}}

\newcommand{\lref}[1]{Lemma~\ref{l.#1}}
\newcommand{\cref}[1]{Corollary~\ref{c.#1}}

\newcommand{\sref}[1]{Section~\ref{s.#1}}

\newcommand{\eref}[1]{(\ref{e.#1})}

\numberwithin{equation}{section}
\numberwithin{theorem}{section}

\newcommand{\Z}{\mathbb{Z}}
\newcommand{\N}{\mathbb{N}}
\newcommand{\R}{\mathbb{R}}

\newcommand{\E}{\mathbb{E}}
\renewcommand{\P}{\mathbb{P}}
\newcommand{\F}{\mathcal{F}}
\newcommand{\Zd}{\mathbb{Z}^d}
\newcommand{\Rd}{\mathbb{R}^d}
\newcommand{\Sy}{\mathbb{S}^d}
\newcommand{\ep}{\varepsilon}
\newcommand{\Pu}{\mathcal{P}}

\newcommand{\parens}[1]{\left( #1 \right)}
\newcommand{\bracks}[1]{\left[ #1 \right]}
\newcommand{\braces}[1]{\left\{ #1 \right\}}

\DeclareMathOperator{\dist}{dist}

\DeclareMathOperator{\var}{var}
\DeclareMathOperator{\cov}{cov}
\DeclareMathOperator{\diam}{diam}
\DeclareMathOperator{\tr}{tr}

\renewcommand{\bar}{\overline}
\renewcommand{\tilde}{\widetilde}

\begin{document}

\title[Quantitative stochastic homogenization]{Quantitative stochastic homogenization of elliptic equations in nondivergence form}

\begin{abstract}
We introduce a new method for studying stochastic homogenization of elliptic equations in nondivergence form. The main application is an algebraic error estimate, asserting that deviations from the homogenized limit are at most proportional to a power of the microscopic length scale, assuming a finite range of dependence. The results are new even for linear equations. The arguments rely on a new geometric quantity which is controlled in part by adapting elements of the regularity theory for the Monge-Amp\`ere equation.

\medskip

\emph{Compared to the published version of this article [\emph{Arch Ration. Mech. Anal.}, \textbf{214}, 867--911 (2014)], this revised version corrects some minor mistakes that were brought to our attention after publication. See Section~\ref{s.changes} below for a description of the changes.}
\end{abstract}

\author[S. N. Armstrong]{Scott N. Armstrong}
\address{Ceremade (UMR CNRS 7534), Universit\'e Paris-Dauphine, Paris, France}
\email{armstrong@ceremade.dauphine.fr}

\author[C. K. Smart]{Charles K. Smart}
\address{Department of Mathematics, Massachusetts Institute of
Technology, Cambridge, MA 02139}
\email{smart@math.mit.edu}

\date{\today}

\maketitle

\section{Introduction}


\subsection{Motivation and informal statements of the main results}

This paper is concerned with uniformly elliptic equations of the form
\begin{equation}\label{e.pde}
F\left(D^2u^\ep(x) ,\frac x\ep\right) = f(x) \quad \mbox{in} \ U \subseteq \Rd,
\end{equation}
where~$F:\Sy\times \Rd\to \R$ is a stationary-ergodic random field, $\Sy$ is the set of $d$-by-$d$ real symmetric matrices and $D^2\phi \in\Sy$ denotes the Hessian matrix of a function $\phi$.

The most important special case of~\eqref{e.pde} is the one in which~$F$ is linear in its first variable and it can be written as 
\begin{equation}\label{e.lin}
-\sum_{i,j=1}^d a_{ij}\left(\frac x\ep \right) \partial_{i}\partial_j u^{\ep}(x) = f(x),
\end{equation}
(for a positive definite matrix~$(a_{ij}(\cdot))_{i,j=1}^d$), which is the master equation governing the behavior of a diffusion in the heterogeneous environment with~covariance matrix~$\sqrt{2a^{ij}}$. General nonlinear equations of the form~\eqref{e.pde} include the Bellman--Isaacs equations, which arise for example in the theory of stochastic optimal control and two-player, zero-sum stochastic differential games.

\smallskip

The essential qualitative result of stochastic homogenization is that the heterogeneous equation~\eqref{e.pde} may be replaced by an \emph{averaged} or \emph{homogenized} one, at least for small $\ep> 0$. More precisely, subject to appropriate boundary conditions, the solutions~$u^\ep$ converge locally uniformly and with probability one, as~$\ep\to 0$, to the  solution~$u$ of the deterministic equation
\begin{equation}\label{e.hom}
\overline F(D^2u) = f(x).
\end{equation}
For the linear equation~\eqref{e.lin}, this was proved more than thirty years ago by Papanicolaou and Varadhan~\cite{PV2} and Yurinski{\u\i}~\cite{Y1}, independently. Their arguments relied on linear duality by considering solutions of the adjoint equation and passing to weak limits to obtain invariant measures. In probabilistic terms, this is often called \emph{the method of the environment from the point of view of the particle} and the homogenization result is formulated as \emph{an invariance principle for a diffusion in a random environment}. The arguments of~\cite{PV2,Y1} can not be generalized to the nonlinear setting, and the general qualitative picture was not completed for another twenty years until Caffarelli, Souganidis and Wang~\cite{CSW} introduced the obstacle problem method, which compares solutions of~\eqref{e.pde} to those of an auxiliary obstacle problem and uses the monotonicity of a certain quantity associated with the latter (the Lebesgue measure of the contact set) to obtain convergence.

\smallskip

From both the theoretical and practical points of view, it is desirable to quantify \emph{just how small}~$\ep$ needs to be in order for~\eqref{e.hom} to be a good approximation of~\eqref{e.pde}.  That is, one would like to describe the distribution of the random field~$u^\ep-u$ and in particular provide upper bounds for its typical size. It is also important to estimate and identify efficient schemes for computing the effective equation $\overline F$. This is the general program of \emph{quantitative} stochastic homogenization. 

\smallskip

The quantitative picture for the stochastic homogenization of nondivergence form equations is incomplete, even for linear equations. In terms of convergence rates to the homogenized limit, it has long been expected that, under an appropriate condition on the random medium quantifying the ergodicity assumption, the typical size of $u^\ep - u$ should be a power of $\ep$, with an estimate like
\begin{equation}\label{e.qwts}
\P \left[ \sup_{x\in U} \left|u^\ep(x) - u(x)\right| \geq C\ep^\alpha \right] \leq C \ep^\beta
\end{equation}
for exponents $\alpha, \beta > 0$ which depend only on the dimension and the ellipticity constant and $C>0$ which may depend also on boundary conditions and other given data. The most natural and important case to address is for random environments satisfying a \emph{finite range of dependence}. The precise definition is given in Subsection~\ref{ssMR} below, but this essentially means that, for some given characteristic length scale~$L>0$, the random variables $F(M,x)$ and $F(M,y)$ are independent whenever $|x-y| \geq L$.

\smallskip

In this paper, we resolve this question by proving~\eqref{e.qwts} in the general nonlinear setting under the assumption of a finite range of dependence. In fact, we prove the stronger estimate: for every $p<d$, there exists $\alpha(p,d,\Lambda)>0$ such that
\begin{equation}\label{e.exp}
\P \left[ \sup_{x\in U} \left|u^\ep(x) - u(x)\right| \geq C\ep^\alpha \right] \leq C \exp\left(- \ep^{-p} \right).
\end{equation}
See Theorem~\ref{t.full} below for the precise statement.  We also prove~\eqref{e.qwts} under appropriate mixing conditions, see Section~\ref{mixing}.

\smallskip

The most significant previous contribution to the theory of quantitative stochastic homogenization of nondivergence form equations is that of~Yurinski{\u\i}~\cite{Y0,Y2}. He proved~\eqref{e.qwts} for the linear equation~\eqref{e.lin} in dimensions five and larger. He also obtained an algebraic error estimate in dimensions three and four in the regime of \emph{small ellipticity contrast} (that is, under the additional and quite restrictive assumption that the diffusion matrix is a very small perturbation of the Laplacian). In dimension two, Yurinski{\u\i}'s arguments give a much slower, logarithmic rate of convergence, even under this assumption. For nonlinear equations, the only quantitative result is due to Caffarelli and Souganidis~\cite{CS}. They proved a logarithmic convergence rate by quantifying the obstacle method proof of convergence from~\cite{CSW}. Precisely, the estimate they get is
\begin{equation}\label{e.obsrate}
\P \left[  \sup_{x\in U} \left|u^\ep(x) - u(x)\right| \geq C\exp \left( -c \sqrt{|\log \ep|} \right) \right]  \leq C  \exp\left( -c \sqrt{|\log\ep|} \right)
\end{equation}
for a constant $c >0$ which depends on dimension and ellipticity and $C>0$ which may depend in addition on the other given data. Since~\cite{CS} seems to optimally quantify the convergence argument of~\cite{CSW}, obtaining the conjectured error estimate~\eqref{e.qwts} requires a different approach to the problem.

\smallskip

In this paper, we introduce a new strategy for studying the homogenization of nondivergence form equations. Since it gives only the second proof of qualitative homogenization for~\eqref{e.pde}, it is of interest beyond the proof of~\eqref{e.exp}. Rather than constraining the solutions via the introduction of an obstacle and measuring the extent to which the solutions feel the constraint, as in~\cite{CS,CSW}, we allow the solutions to be free and measure the curvature of their graphs. This curvature is captured by a new monotone quantity, denoted below by~$\mu(U,F)$, which measures how many planes may touch a supersolution of $F(D^2u,x) \geq 0$ from below in $U$. At the core of our approach are the results in Section~\ref{sec.SC}, which assert that solutions which maximize this curvature, in this sense, must be uniformly convex (in proportion to the curvature). The proof of this uses geometric ideas inspired by the regularity theory of the Monge-Amp\`ere equation. This connection arises naturally from the quantity~$\mu$ itself: see Lemma~\ref{l.mooney} and the comments preceding it, as well as the discussion in Subsection~\ref{ss.muintro}.

\smallskip

Most of the work for proof the main error estimates lies in obtaining an appropriate  estimate on the decay of $\mu(Q,F)$ as the cube $Q$ becomes large. This is stated in Theorem~\ref{t.mudecay}, below, and the focus of most of the paper. Once we have obtained this estimate, the main result follows by showing that $\mu$ controls the difference $u^\ep-u$ of the solutions of the corresponding Dirichlet problems. This is accomplished through a relatively straightforward comparison argument quantified by the regularity theory.

\subsection{Hypotheses and review of qualitative results}

Before stating the main result, we introduce the notation, give the precise assumptions, and review the qualitative theory.

\smallskip

Throughout the paper, we work in~$\Rd$ in dimension~$d\geq 2$ and all differential equations and inequalities are to be interpreted in the viscosity sense (c.f.~\cite{CC,CIL}). The set of real $d$-by-$d$ symmetric matrices is denoted by $\Sy$. If $A\in \Sy$, then $|A|$ denotes the square root of the largest eigenvalue of $A^2$. We write $A\geq 0$ if $A$ has nonnegative eigenvalues. Recall that the Pucci extremal operators with ellipticity $\Lambda >1$ are defined for each $A\in \Sy$ by:
\begin{equation*} \label{}
\Pu^+_{1,\Lambda}(A) = -\tr(A_+) + \Lambda \tr(A_-)\quad \mbox{and} \quad \Pu^-_{1,\Lambda}(A) = -\Lambda \tr(A_+) +  \tr(A_-)
\end{equation*}
Here $\tr (A)$ denotes the trace of $A$, and $A_+,A_-\in \Sy$ are the uniquely defined by the conditions: $A=A_+-A_-$, $A_+A_- = 0$ and $A_+,A_-\geq 0$. The identity matrix is denoted by~$I$.

\begin{definition}
\label{d.Omega}
Given $\Lambda > 1$, we take $\Omega$ to be the set of ``all uniformly elliptic equations with ellipticity~$\Lambda$." Precisely, we consider functions $$F:\Sy\times \Rd \to \R$$ which satisfy the following conditions: for every $A,B\in \Sy$ and $x\in \Rd$, 
\begin{equation}\label{e.Fue}
\Pu_{1,\Lambda}^-(A-B) \leq F(A,x)-F(B,x) \leq \Pu_{1,\Lambda}^+(A-B) \quad \mbox{(uniform ellipticity),}
\end{equation}
there exist constants $C>0$ and $\frac12 < \theta \leq 1$ such that, for all $A\in \Sy$ and $x,y\in \Rd$,
\begin{equation}\label{e.Fsr}
|F(A,x)-F(A,y)| \leq  C\left(1+|A|\right) |x-y|^{\theta} \quad \mbox{(spatial regularity)}
\end{equation}
and
\begin{equation}\label{e.Fbdd}
\sup_{x\in \Rd} |F(0,x)| < +\infty \quad \mbox{(boundedness),}
\end{equation}
and we define 
\begin{equation*}\label{}
\Omega:=\Omega(\Lambda):= \left\{ F:\Sy\times \Rd\to \R \ \ \mbox{satisfies~\eqref{e.Fue}, \eqref{e.Fsr} and~\eqref{e.Fbdd}} \right\}.
\end{equation*}
We endow~$\Omega$ with the $\sigma$--algebra~$\F$, given by
\begin{equation*}\label{}
\F:= \mbox{$\sigma$--algebra on $\Omega$ generated by the family $\left\{ F \mapsto F(A,x)\,:\, (A,x) \in \Sy\times \Rd \right\}$}.
\end{equation*}
We denote the set of constant-coefficient operator by $\overline \Omega:=\overline\Omega(\Lambda)\subseteq \Omega$, that is, the set of~$F$'s which do not depend on the second variable. 
\end{definition}
We remark that the purpose of the hypothesis~\eqref{e.Fsr} is to ensure that the comparison principle holds (c.f~\cite{CIL}). It is irrelevant how small $\theta -\frac12$ may be or how large $C$ is in this inequality in the sense that none of our quantitative estimates depend on these parameters. 

The \emph{random environment} is modeled by a probability measure~$\P$ on~$(\Omega,\F)$, which is assumed to have the following properties. First, there exists~$K_0>0$ such that~$\P$ is supported on the set of~$F$ for which~$|F(0,\cdot)|$ is uniformly bounded by~$K_0$; i.e.,
\begin{equation}\label{e.Fubb}
\P \left[ \ \sup_{x\in \Rd} \left| F(0,x) \right| \leq K_0 \right] = 1 \qquad \mbox{(uniform boundedness).}
\end{equation}
Next, $\P$ is assumed to be \emph{stationary}, i.e., it is invariant under  translations. Denote the action of translation by $T: \Rd \times \Omega\to \Omega$,
\begin{equation*}
T(y,F)(A,x) := (T_yF)(A,x) := F(A,x+y),
\end{equation*}
and extend this to $\F$ by setting $T_yE:= \{ T_yF\,:\, F\in E\}$ for $E\in \F$. Stationarity is the assumption that
\begin{equation}\label{e.statRd}
\forall E\in\F,\, \forall y\in \Rd: \qquad \P\left[ T_y E \right] = \P\left[ E \right] \qquad \mbox{(stationarity).}
\end{equation}
For the most general qualitative theory of stochastic homogenization, the natural condition to impose on $\P$, in addition to stationarity, is \emph{ergodicity}: this means that the only events which are translation invariant are those of null or full probability. The precise hypothesis is that, for every $E\in \F$,
\begin{equation}\label{e.erg}
E=\bigcap_{y\in \Rd} T_yE  \qquad \mbox{implies that} \qquad \P\left[E\right] \in \{ 0 ,1 \} \qquad \mbox{(ergodicity).}
\end{equation}
We denote by $\E$ the expectation with respect to $\P$.

\smallskip

We next recall the statement of qualitative stochastic homogenization, presented in terms of solutions of the Poisson-Dirichlet problem for $F$.
\addtocounter{theorem}{-1}

\begin{theorem}[Linear case:~\cite{PV2,Y1}, full generality:~\cite{CSW}]
\label{t.csw}
Fix~$\Lambda > 1$ and~$K_0>0$, and assume that~$\P$ is a probability measure on~$(\Omega(\Lambda),\F)$ satisfying~\eqref{e.Fubb},~\eqref{e.statRd} and~\eqref{e.erg}.
Then there exist $\overline F \in \overline\Omega(\Lambda)$ and~$\Omega_0 \in \F$ with $\P\left[ \Omega_0 \right] =1$ such that the following holds: for every $F\in \Omega_0$, bounded smooth domain $U\subseteq \Rd$, $g \in C(\partial U)$, and $f \in C(U) \cap L^\infty(U)$, the unique solution $u^\ep(\cdot,F) \in C(\overline U)$, for $\ep > 0$, of the Dirichlet problem
\begin{equation}\label{e.cswuep}
\left\{ \begin{aligned}
& F\left(D^2u^\ep,\frac x\ep\right) = f & \mbox{in} & \ U,\\
& u^\ep = g & \mbox{on} & \ \partial U,
\end{aligned} \right.
\end{equation}
satisfies
\begin{equation*}\label{}
\lim_{\ep \to 0} \sup_{x\in U} \big| u^\ep(x,F) - u(x) \big| = 0,
\end{equation*}
where $u \in C(\overline U)$ denotes the unique solution of
\begin{equation}\label{e.cswu}
\left\{ \begin{aligned}
& \overline F(D^2u) = f & \mbox{in} & \ U,\\
& u = g & \mbox{on} & \ \partial U.
\end{aligned} \right.
\end{equation}
\end{theorem}

The qualitative homogenization results stated in~\cite{CSW} are more general than what we have presented above, and include equations with lower-order dependence, mild coupling between the microscopic and macroscopic scales and results for time-dependent parabolic problems. The decision to state Theorem~\ref{t.csw} as well as the main result, Theorem~\ref{t.full}, in terms of the Poisson-Dirichlet problem and with less than full generality is not due to any limitations of our method: essentially all of the difficulty lies in proving this case and the desired extensions and generalizations are fairly straightforward to obtain. See \sref{remarks} for more discussion.

Although we focus on obtaining quantitative results and thus assume stronger hypotheses, a new proof of Theorem~\ref{t.csw} can also be extracted from the arguments in this paper.

\subsection{Statement of the main result}
\label{ssMR}

For quantitative results, it is necessary to add an assumption which quantifies the ergodicity of~$\P$. In this paper, we postulate that~$\P$ enjoys a \emph{finite range of dependence}.  Let us give the precise statement of this hypothesis. We first denote, for each Borel set $U\subseteq \Rd$,
\begin{equation}\label{e.FU}
\F(U):= \mbox{$\sigma$--algebra on $\Omega$ generated by} \ \ \{ F \mapsto F(A,x) \,:\, (A,x) \in \Sy\times U \}.
\end{equation}
Intuitively, we think of~$\mathcal F(U)$ as containing the information about the behavior of the random environment in~$U$. Note that $\F(U)\subseteq \F = \F({\Rd})$. The finite range of dependence condition is stated as follows: 
\begin{multline}\label{e.frd}
\mbox{for all Borel subsets} \ U,V \subseteq \Rd \ \ \mbox{such that} \ \dist(U,V) \geq 1, \\
\F(U) \   \mbox{and} \ \F(V)\ \ \mbox{are \ $\P$--independent}.
\end{multline}
Here $\dist(U,V) := \inf\{  |x-y|\,:\, x\in U, \ y\in V \}$ denotes the usual distance between subsets of $\Rd$. Note that~\eqref{e.frd} implies~\eqref{e.erg}.

In order to include important examples such as the random checkerboard, we relax the stationary hypotheses to the assumption that $\P$ is invariant under \emph{integer} translations. Instead of~\eqref{e.statRd}, we require 
\begin{equation}\label{e.stat}
\forall E\in\F,\, \forall z\in \Zd: \qquad \P\left[ T_z E \right] = \P\left[ E \right] \qquad \mbox{(stationarity).}
\end{equation}

\smallskip

We next present the main result.

\begin{theorem}
\label{t.full}
Suppose $\Lambda > 1$, $K_0>0$, and $\P$ is a probability measure on~$(\Omega(\Lambda),\F)$ satisfying~\eqref{e.Fubb},~\eqref{e.frd} and~\eqref{e.stat}. If $0< p < d$, $0< \ep \leq 1$, $U \subseteq \R^d$ is a bounded smooth domain, $g \in C^{0,1}(\partial U)$, $f \in C^{0,1}(U)$, $u^\ep(\cdot, F) \in C(\bar U)$ denotes the unique solution of \eref{cswuep}, and $u \in C(\bar U)$ denotes the unique solution of \eref{cswu}, then
\begin{equation*}
\P\bracks{ \sup_{x \in U} |u^\ep(x,F) - u(x)| \geq C\ep^\alpha} \leq C\exp(- \ep^{-p}),
\end{equation*}
for some exponent $\alpha>0$ satisfying $\alpha \geq c_0(d-p)$ for $c_0(d,\Lambda)>0$ and a constant $C > 0$ depending only on $d$, $\Lambda$, $K_0$, $U$, $p$, $\| g \|_{C^{0,1}(\partial U)}$, and $\| f \|_{C^{0,1}(U)}$.
\end{theorem}

See Theorem~\ref{t.mixing} for an extension of Theorem~\ref{t.full} to probability measures satisfying a uniform mixing condition rather than~\eqref{e.frd}.

\subsection{Outline of the paper}

In the next section we introduce the new monotone quantity $\mu(U,F)$ and review some of its elementary properties. We also give the statements of Theorem~\ref{t.mudecay} and Corollary~\ref{c.mudeviation}, which provide a strong algebraic rate of decay for $\mu$. The next three sections are devoted to the proofs of these results. In Section~\ref{S.full}, we obtain~Theorem~\ref{t.full} from Theorem~\ref{t.mudecay} via a quantitative comparison argument. We conclude in Section~\ref{s.remarks} with some remarks and open problems.

\subsection{Description of changes from the published version}
\label{s.changes}

As explained below the abstract, this version of the paper is a revision made more than five years after publication. Three years ago, it was brought to 
our attention by Xiaoqin Guo that the proof of the second statement of Lemma~2.5 in the published version was incorrect. Since the fix, which consists in deleting the second part of Lemma~2.5 and using a weaker form of Lemma~2.8, is easy, we believed that an errata was not warranted and we were content to explain the fix privately. Since that time, there have been some further developments on the topic and some doubts concerning the arguments in the published version of this article have recently come to our attention. We have therefore come to the conclusion that a public posting is now needed, demonstrating that the original arguments stand, with minor corrections. The purpose of the current revision is to address this. We also have fixed a few other minor glitches, some of which were pointed out to us by Yves Capdeboscq.

\smallskip

The main changes from the published version are as follows. The statements and proofs of Lemmas~\ref{l.mugood} and~\ref{l.mubalance} have been modified. The first paragraph of the proof of Theorem~\ref{t.mudecay} has been reworded to reflect the change to Lemma~\ref{l.mubalance}. A minor glitch in Step~4 in the proof of Proposition~\ref{p.snapgrid} has been fixed. We have also added some explanations to some of the arguments which were a bit quick in the previous version (such as in Step~2 of the proof of Lemma~\ref{l.contract}). Several other scattered typos were corrected, such as in~\eqref{e.cquasi1} and~\eqref{e.cquasi2}.

\section{A new monotone quantity}
\label{S.muintro}

We introduce $\mu(U,F)$, derive some of its properties and give the statements of the main results concerning its decay (Theorem~\ref{t.mudecay} and Corollary~\ref{c.mudeviation}). 

\subsection{The definition of $\mu$}
\label{ss.muintro}
We begin with some notation. Given $F \in \Omega$ and a bounded open set $U \subseteq \R^d$, let
\begin{equation*}
S(U,F) := \left\{ u \in C(\overline U) \,:\, F(D^2 u, x) \geq 0 \ \ \mbox{in} \ U \right\},
\end{equation*}
denote the set of supersolutions of $F$ in $U$ that are continuous on~$\overline U$. The {\em convex envelope} of a function $u \in C(U)$ is denoted by
\begin{equation*}
\Gamma_u(x) := \sup_{p \in \R^d} \inf_{y \in U} (u(y) + p \cdot (x - y)).
\end{equation*}
Although $\Gamma_u$ depends on $U$, we do not display this dependence. Given a function $w \in C(U)$ and $x \in U$, the {\em subdifferential} of $w$ at $x$ is denoted by
\begin{equation*}
\partial w(x) := \left\{ p \in \R^d \,:\, w(y)  \geq w(x) + p \cdot (y-x) \quad \mbox{for all}  \  y \in U \right\}
\end{equation*}
and, for each $V\subseteq U$, we denote the image of $V$ under $\partial w$ by 
\begin{equation*}\label{}
\partial w(V) := \bigcup_{x\in V} \partial w(x).
\end{equation*}
We now define, for every $F\in \Omega$ and bounded domain $U\subseteq \Rd$, the quantity
\begin{equation*}\label{}
\mu(U,F):= \frac{1}{|U|} \sup\big\{ \left| \partial \Gamma_u(U)\right|\,:\, u \in  S(U,F) \big\}.
\end{equation*}
Here and throughout the paper, $|E|$ denotes the Lebesgue measure of $E \subseteq\Rd$. 

\smallskip

To get a rough geometric idea of what exactly $\mu$ is measuring, notice that if $F \in \overline\Omega$ is a constant coefficient operator, then $\mu$ has the following simple form:
\begin{equation*}\label{}
\mu(U,F) = \mu(F)= \sup\left\{ \det A\,:\, A\in \Sy, \ A \geq 0, \ F(A) \geq 0 \right\}.
\end{equation*}
In other words, if $F$ is independent of $x$, then an optimizer~$u$ in the definition of~$\mu$ is a simultaneous solution of $F(D^2u) = 0$ and the Monge-Amp\`ere equation $\det D^2u = k$, with the largest possible $k>0$. In the general case, assuming $\Gamma_u$ has enough regularity (and we will see below that it does), the area formula for Lipschitz functions permits us to write

\begin{equation*}\label{}
\mu(U,F) = \sup\left\{ \fint_{U} \det D^2 \Gamma_u(x)\,dx \,: \, u\in C(\overline U) \ \mbox{satisfies}\ F(D^2u,x) \geq 0 \ \mbox{in} \ U \right\}.
\end{equation*}
(Here and throughout, we denote the average over $E$ by $\fint_E$, that is, $\fint_E f(x)\,dx  : = |E|^{-1} \int_E f(x)\,dx$.) Thus $\mu$ is an affine-invariant quantity which measures how much curvature the graph of the convex envelope of a solution of $F=0$ may have.

\smallskip

We remark that in the case $F(0,x) \leq 0$, we trivially have $\mu(U,F) =0$ by the maximum principle. The definition of $\mu$ may therefore seem strange to a reader who has in mind a linear operator. This confusion disappears in view of the fact that, in the proof of Theorem~\ref{t.full}, we apply the estimates for the quantity $\mu$ obtained in the next two sections not only to $F$ but to all translations of $F$ by paraboloids, i.e., to all operators of the form $F_A(B,x):= F(A+B,x)$. The reason for suppressing the dependence on $A$ at this stage can be found in~Subsection~\ref{ss.pushfor}.

\subsection{The definition of $\mu_*(U,F)$}

As we will see below, the quantity $\mu(U,F)$ controls solutions of $F=0$ from below. In order to control solutions from above, we introduce the twin of $\mu(U,F)$, which we denote by $\mu_*(U,F)$. Before giving its definition, we first define an involution $F\mapsto F_*$ on $\Omega$ by
\begin{equation*}\label{}
F_*(A,x):= -F(-A,x),\quad (A,x) \in \Sy\times \Rd.
\end{equation*}
One can check that the map $F \mapsto F_*$ is indeed a bijection from $\Omega(\Lambda)$ to itself, and $F_{**} = F$. The usefulness in considering $F_*$ is due to the fact that, for each $u\in C(\overline U)$,
\begin{equation}\label{e.odd.dual}
u\in S(U,F_*) \quad \iff \quad v:= -u \quad \mbox{satisfies} \quad F(D^2v,x) \leq 0 \quad \mbox{in} \ U,
\end{equation}
which is an immediate consequence of the viscosity solution definitions.

We define, for every bounded domain $U\subseteq\Rd$,
\begin{align}\label{e.invs}
\mu_*(U,F)  : =
& \ \frac{1}{|U|} \sup\big\{ \left| \partial \Gamma_{-u}(U) \right| \,: \,  \mbox{$u \in C(\overline U)$ satisfies $F(D^2u,x) \leq 0$ in $U$} \big\} \\ = & \
\frac{1}{|U|} \sup\big\{ \left| \partial \Gamma_{u}(U) \right| \,: \,  \mbox{$u\in S(U,F_*)$ } \big\} = \mu(U,F_*).\nonumber
\end{align}
In short, the quantity $\mu_*$ is the analogue of $\mu$ for subsolutions of $F(D^2u,x) = 0$ rather than supersolutions. Often we write $\mu(U,F_*)$ in place of $\mu_*(U,F)$.

\subsection{Pushforwards of $\P$}
\label{ss.pushfor}
Recall that if $\pi :\Omega \to \Omega$ is an $\F$--measurable map, then the \emph{pushforward of $\P$ under $\pi$} is the probability measure $\pi_{\#}\P$ defined by
\begin{equation*} \label{}
\pi_{\#}\P\left[E\right]:= \P\left[ \pi^{-1}(E) \right].
\end{equation*}
The pushforward of $\P$ under the involution $F \mapsto F_*$ 
enjoys the same hypotheses as $\P$. Therefore, in view of~\eqref{e.invs}, many assertions we make concerning~$\mu$ have analogous formulations in terms of~$\mu_*$. Similarly, for every $s \in \R$, the pushforward of~$\P$ under the shift map $F \mapsto F + s$, where $(F+s)(A,x) := F(A,x) + s$, also preserves the hypotheses except that the constant $K_0$ in~\eqref{e.Fubb} must be replaced by $K_0+|s|$. Likewise, for $A\in \Sy$, the pushforward of~$\P$ under the translation $F\mapsto F_A$, given by
\begin{equation}\label{e.Ftrans}
F_A(B,x):= F(A+B,x), \quad (B,x) \in \Sy \times \Rd,
\end{equation}
also satisfies the same hypotheses as $\P$, after we replace $K_0$ by $K_0+d\Lambda |A|$.

\subsection{A triadic cube decomposition}
Throughout the paper, we work with the following triadic cube decomposition. For every $m\in \Z$, we set 
\begin{equation*}
Q_m :=  \left(- \tfrac{1}{2} 3^m, \tfrac{1}{2} 3^m \right)^d
\end{equation*}
and, for every $x\in \Rd$, we denote
\begin{align*}\label{}
Q_m(x)  :=
3^{m} \left\lfloor 3^{-m} x+ \tfrac12 \right\rfloor  + Q_m.
\end{align*}
Here $\left\lfloor r \right\rfloor$ denotes, for $r\in \R$, the largest integer not larger than $r$ and we write $\lfloor y \rfloor  := \left( \lfloor y_i \rfloor\right)$ for $y=(y_i)\in \Rd$. Up to a set of zero Lebesgue measure, $Q_m(x)$ is the unique cube of the form $3^mk + Q_m$, with $k\in \Z^d$, containing $x$. Also note that, up to a zero measure set, $Q_m(x) = Q_m(y)$ if and only if $x\in Q_m(y)$. In particular, the cube $Q_m(x)$ is \emph{not} the centered at $x$ unless $x\in 3^m\Zd$. There are exactly $3^{d(m-n)}$ cubes of the form $Q_n(x)$ with $Q_n(x) \subseteq Q_m$, and these form an exact partition of~$Q_m$ up to a set of Lebesgue measure zero. 

\smallskip

For every $m\in \Z$, $\{ Q_m(x)\,:\, x\in \Rd\}$ is a pairwise disjoint partition of $\Rd$, up to a set of zero Lebesgue measure. Likewise, for each $m\in \Z$ and $n\in \N$,  $\{ Q_m(x)\,:\, x\in Q_{m+n}\}$ is a pairwise disjoint partition of $Q_{m+n}$ into $3^{d n}$ distinct subcubes, up to a zero measure set. 

Note that, for every $m\in \N$ and $x,y\in \Rd$, the cubes $Q_m(x)$ and $Q_m(y)$ are integer translations of each other, and therefore~\eqref{e.stat} implies for example that the random variables $\mu(Q_m(x),F)$ and $\mu(Q_m(y),F)$ have the same distribution under $\P$.

It is often notationally convenient to express sums over our triadic cubes as integrals: for example, we may write
\begin{equation*}\label{}
\sum_{\{ Q \,:\, Q = Q_m(x) \subseteq Q_{m+n}\}} \mu(Q,F) = \frac{1}{|Q_m|}\int_{Q_{m+n}} \mu(Q_m(x),F) \, dx.
\end{equation*}

\subsection{Basic properties of $\mu$}

We begin by showing that $\mu$ controls supersolutions from below. 

\begin{lemma}
\label{l.muabp}
There is a constant $C(d) > 0$ such that, for every $F\in \Omega$, $x\in \Rd$, $m\in \Z$ and $u\in S(Q_m(x),F)$,
\begin{equation}\label{e.muabp}
\inf_{\partial Q_m(x)} u  \leq \inf_{Q_m(x)} u + C 3^{2m}  \mu(Q_m(x),F)^{1/d}.
\end{equation}
\end{lemma}
\begin{proof}
By translating and rescaling, we may suppose that $x=0$ and $m=0$. We may also assume that $a:= \inf_{\partial Q_0} u - \inf_{Q_0} u>0$, since otherwise there is nothing to show. Select $x_0\in Q_0$ such that $\inf_{Q_0} u = u(x_0)$. For every $p\in \Rd$ such that $|p| < a(\diam(Q_0))^{-1} = ad^{-1/2}$, we have
\begin{equation*}\label{}
u(x_0) - p\cdot x_0 = \inf_{\partial Q_0} u - a - p\cdot x_0 \leq \inf_{y \in \partial Q_0} \left( u(y) - p\cdot y \right) \underbrace{-a + |p|\diam(Q_0)}_{\leq 0}.
\end{equation*}
Hence for any such $p$, the map $x\mapsto u(x) - p\cdot x$ achieves its infimum with respect to $Q_0$ at some point of $Q_0$ and thus $p \in \partial \Gamma_u(Q_0)$. We deduce that $B_{ad^{-1/2}} \subseteq \partial \Gamma_u(Q_0)$. In particular, 
\begin{equation*}\label{}
\left| \partial \Gamma_u(Q_0) \right| \geq \left| B_{ad^{-1/2}}\right| = |B_1| \left( ad^{-1/2} \right)^d.
\end{equation*}
Rearranging and using $u\in S(Q_0,F)$, we obtain
\begin{equation*}\label{}
a \leq |B_1|^{-1/d} d^{1/2} \left( \frac{\left| \partial \Gamma_u(Q_0) \right|}{|Q_0|} \right)^{1/d} \leq C \mu(Q_0,F)^{1/d}.\qedhere
\end{equation*}
\end{proof}

The following lemma is a variation on Lemma 3.3 of \cite{CC}.  We include a proof for completeness and the reader's convenience.

\begin{lemma}
\label{l.envelopec11}
Suppose that $U \subseteq \R^d$ is open, $B_R\subseteq U$ and $u \in C(\overline U)$ satisfies 
\begin{equation*} \label{}
\Pu^+_{1,\Lambda}(D^2 u) \geq -1 \quad \mbox{in} \ U.
\end{equation*}
Then there exists $C(d,\Lambda)>0$ such that, for every $x_0 \in \{ x\in U:\Gamma_u(x) = u( x) \}$, $p\in \partial \Gamma_u(x_0)$ and $0<4r<R$, 
\begin{equation*}
\partial \Gamma_u\left(B_r(x_0)\right) \subseteq B_{2 r + CR^{-2} r^3}(p).
\end{equation*}
\end{lemma}

\begin{proof}
We may assume $x_0=0$ and, by subtracting a plane from $u$, that $p=0$ and $u(0) = 0$. By a scaling argument, it suffices to consider the case $R = 4$ and $0<r<1$ and to prove, for some $C(d,\Lambda)>0$, that
\begin{equation} \label{e.pects}
\partial \Gamma_u(B_r) \subseteq B_{2r + Cr^3}.
\end{equation}
We suppose that $q \in \partial \Gamma_u(B_r)$ and $|q| \geq 2 (1 + \delta)^3 r$ for some $0<\delta< 1$ and endeavor to prove an upper bound on $\delta$.  By rotating the coordinates, we may assume that $q = |q| e_1$. We get 
\begin{equation} \label{e.smects}
u \geq \Gamma_u \geq 2 (1 + \delta)^3 r \max\{ 0, e_1 \cdot x - r \} \quad \mbox{in } B_4.
\end{equation}
Let $S$ denote the cylinder
\begin{equation*}
S := (-2\delta r,2 r) \times B_1',
\end{equation*}
where $B_1'$ denotes the unit ball in $\R^{d-1}$. Consider test function
\begin{equation*}
\varphi(x) := \frac{(1 + \delta)}{2} (e_1 \cdot x + 2 \delta r)^2 -\frac{\delta}{2\Lambda(d-1)} |x - (e_1 \cdot x) e_1|^2.
\end{equation*}
After a computation, we find that
\begin{equation*} \label{}
\Pu^+_{1,\Lambda}(D^2 \varphi) = -1 \quad \mbox{in} \ \Rd.
\end{equation*}
Since $\varphi(0) \geq 0 = u(0)$ and $S \subseteq B_4$, the comparison principle implies that 
\begin{equation*} \label{}
\inf_{\partial S} (\Gamma_u - \varphi) \leq \inf_{\partial S} (u - \varphi) \leq 0.
\end{equation*}
Using~\eqref{e.smects} and $0 < \delta < 1$, it is straightforward to check that 
\begin{equation*}
\varphi \leq 2 (1 + \delta)^3 r^2 \leq \Gamma_u \quad \mbox{on } \{ 2 r \} \times B_1'
\end{equation*}
and
\begin{equation*}
\varphi \leq 0 \leq \Gamma_u \quad \mbox{on } \{ - 2 \delta r \} \times B_1'.
\end{equation*}
We are forced to conclude, using~\eqref{e.smects} and the definition of $\varphi$, 
\begin{equation*}
0 \geq \inf_{\left(-2\delta r,2r \right) \times \partial B_1'} (\Gamma_u - \varphi) \geq - \sup_{\left(-2\delta r,2r \right) \times \partial B_1'} \varphi \geq - 2(1 + \delta)^3 r^2 + \frac{\delta}{2 \Lambda (d - 1)}.
\end{equation*}
Rearranging and using $0 < \delta < 1$, we get
\begin{equation*} \label{}
\delta \leq 4 \Lambda(d-1)(1 + \delta)^3 r^2 \leq 32 \Lambda (d - 1) r^2.
\end{equation*}
This holds for all $0< \delta < 1$ such that $|q| \geq 2(1+\delta)^3 r$, and from this we obtain that  $|q| \leq 2 r + C(d,\Lambda) r^3$. This yields~\eqref{e.pects} and completes the proof.
\end{proof}

Since convex solutions of $\Pu^+_{1,\Lambda}(D^2 u) \geq -1$ satisfy $0 \leq D^2 u \leq I$, one might expect that the optimal estimate in \lref{envelopec11} should be: $\partial \Gamma_u(B_r(x_0)) \subseteq B_{r + o(r)}(p)$. Interestingly, this turns out to be false and Lemma~\ref{l.envelopec11} is actually optimal, as the following example shows.

\begin{example}
\label{ex.twodahman}
Given $R > 1$, let $u, w \in C(\R^2)$ and $U \subseteq \R^2$ be defined by
\begin{equation*}
u(x) = \tfrac{1}{2} x_1^2 - \tfrac{1}{2R} \max \{ 0, |x_2| - R \}^2,
\end{equation*}
\begin{equation*}
w(x) = 2 \max \{ 0, x_1 - 1 \} + \tfrac{1}{2} \max \{ 0, |x_1 - 1| - 1 \}^2,
\end{equation*}
and
\begin{equation*}
U = [-R,R]^2 \cup \{ u > w \}.
\end{equation*}
The domain $U$ and the cross sections $u(\cdot,0)$ and $w(\cdot,0)$ are pictured in Figure~\ref{figfig}. One can check $B_R \subseteq U \subseteq [-2R,2R]^2$, $\Pu^+_{1,\Lambda}(D^2 u) \geq -1$ in $U$,  $w$ is the convex envelope of $u$ with respect to the domain $U$, $0 \in \partial w(0)$, and $2e_1 \in \partial w(e_1)$.
\end{example}

\begin{figure}
\begin{tikzpicture}[thick,scale=0.8]
\draw (2,0) -- (3,0.5) -- (2,1) -- (2,2) -- (-2,2) -- (-2,1) -- (-3,0.5) -- (-2,0) -- (-2,-2) -- (2,-2) -- (2,0);
\draw[dotted] (2,0) -- (-2,0);
\draw[dotted] (2,1) -- (-2,1);
\draw (0,0.5) node {$\{ u > w \}$};
\end{tikzpicture}
\qquad
\begin{tikzpicture}[thick]
\draw (-2.50,3.12) --(-2.40,2.88) --(-2.30,2.64) --(-2.20,2.42) --(-2.10,2.21) --(-2.00,2.00) --(-1.90,1.80) --(-1.80,1.62) --(-1.70,1.44) --(-1.60,1.28) --(-1.50,1.12) --(-1.40,0.98) --(-1.30,0.85) --(-1.20,0.72) --(-1.10,0.61) --(-1.00,0.50) --(-0.90,0.41) --(-0.80,0.32) --(-0.70,0.24) --(-0.60,0.18) --(-0.50,0.12) --(-0.40,0.08) --(-0.30,0.04) --(-0.20,0.02) --(-0.10,0.01) --(0.00,0.00) --(0.10,0.01) --(0.20,0.02) --(0.30,0.04) --(0.40,0.08) --(0.50,0.12) --(0.60,0.18) --(0.70,0.24) --(0.80,0.32) --(0.90,0.41) --(1.00,0.50) --(1.10,0.61) --(1.20,0.72) --(1.30,0.85) --(1.40,0.98) --(1.50,1.12) --(1.60,1.28) --(1.70,1.44) --(1.80,1.62) --(1.90,1.80) --(2.00,2.00) --(2.10,2.21) --(2.20,2.42) --(2.30,2.64) --(2.40,2.88) --(2.50,3.12);
\draw (0,0) -- (1,0) -- (2,2);
\fill (0,0) circle (2pt);
\fill (1,0) circle (2pt);
\fill (2,2) circle (2pt);
\draw (0.5,0.5) node {$u$};
\draw (1.5,0.25) node {$w$};
\end{tikzpicture}
\caption{The domain $U$ and cross sections $u(\cdot,0)$ and $w(\cdot,0)$ in Example~\ref{ex.twodahman}.}
\label{figfig}
\end{figure}
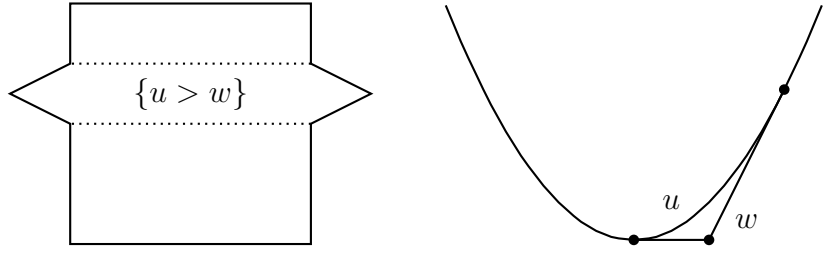

Example~\ref{ex.twodahman} serves as a warning that convex envelopes of supersolutions may not be so regular: singularities may can propagate inward from the boundary. The next lemma asserts that such singularities have no impact on the Lebesgue measure of the subdifferential of the convex envelope. The idea for this estimate was extracted from the proof of Lemma 3.5 in \cite{CC}.  However, the statement  here is more general (it does not require that $u$ be a supersolution) and the extra generality leads to a simpler proof.

\begin{lemma}
\label{l.envelopefacet}
Assume that $U \subseteq \R^d$ is open, $u \in C(\overline U)$, $x_0\in U$ and $r>0$ such that
\begin{equation*} \label{}
B_r(x_0) \subseteq \left\{ x\in U\,:\, \Gamma_u(x) < u(x)\right\}.
\end{equation*}
Then $\left|\partial \Gamma_u(B_r(x_0))\right| = 0$.
\end{lemma}

\begin{proof}
By a covering argument, it is enough to show that $|\partial \Gamma_u(B_r(x_0))| = 0$ in the case that $B_{3r}(x_0) \subseteq  \{ u > \Gamma_u \}$.  Arguing by contradiction, we suppose that $B_{3r}(x_0) \subseteq \{ u > \Gamma_u \}$, $x_1 \in B_r(x_0)$, $p_1 \in \partial \Gamma_u (x_1)$, and $p_1$ is a Lebesgue density point of $\partial \Gamma_u (B_r(x_0))$. By translating and adding an affine function to $u$, we may assume that $x_1 = 0$, $p_1 = 0$ and $\Gamma_u(0) = 0$.  In particular, we have $B_{2r} \subseteq B_{3r}(x_0) \subseteq \{ u > \Gamma_u \}$, $0 \in \partial \Gamma_u(0)$, and $0$ is a Lebesgue density point of~$\partial \Gamma_u(B_r)$.

Using that $0$ is a Lebesgue density point of $\partial \Gamma_u(B_r)$, for any given $x \in \partial B_r$, there exists $q \in \partial \Gamma_u(B_r) \setminus \{ 0 \}$ such that
\begin{equation*}
q \cdot x \geq \frac{3}{4} |x| |q|.
\end{equation*}
Let $y \in B_r$ be such that $q \in \partial \Gamma_u(y)$.  Taking $\alpha \geq 2$ such that $\alpha x \in \overline U$ and using that $\Gamma_u\geq 0$, we obtain
\begin{equation*}
\Gamma_u(\alpha x) \geq \Gamma_u(y)  + q \cdot (\alpha x - y)  \geq \alpha q \cdot x - q \cdot y \geq \frac{3}{4} \alpha r |q| - r |q| > 0.
\end{equation*}
Since this holds for all $x\in \partial B_r$, we deduce that
\begin{equation*}
\Gamma_u > 0 \quad \mbox{on} \ \overline U \setminus B_{2r}.
\end{equation*}
Since $0 \in \partial \Gamma_u(0)$ and $\Gamma_u(0)=0$ imply that $\inf_{B_{2r}} \Gamma_u = 0$, and using the fact that $u > \Gamma_u$ on $\overline B_{2r}$, we conclude that
\begin{equation*}
\inf_{U} u > 0.
\end{equation*}
This contradicts $\Gamma_u(0) = 0$, by the definition of convex envelope.
\end{proof}

We next combine Lemmas~\ref{l.envelopec11} and~\ref{l.envelopefacet} to get the boundedness and Lipschitz continuity (with respect to perturbing by parabolas) of the Lebesgue measure of the subdifferential of a supersolution.

\begin{lemma}
\label{l.mugood}
Assume $U \subseteq \R^d$ is bounded and open and $u \in C(U)$ satisfies
\begin{equation*}
\Pu^+_{1,\Lambda}(D^2u) \geq -1 \quad \mbox{in } U.
\end{equation*}
Then
\begin{equation} \label{e.sharpie}
|\partial \Gamma_u(U)| \leq 2^d \left|\left\{ x\in U\,:\, u(x) = \Gamma_u(x) \right\} \right|.
\end{equation}
\end{lemma}

\begin{proof}
\lref{envelopec11} implies that $\partial \Gamma_u(x)$ is a singleton set, for every $x\in \{ u = \Gamma_u\}$. Since $\Gamma_u$ is convex, this implies that $\Gamma_u$ is differentiable on $\{ u = \Gamma_u\}$. Using again Lemma~\ref{l.envelopec11}, the Lebesgue differentiation theorem and a covering argument, we obtain  
\begin{equation*} \label{}
\left| \partial \Gamma_u(\{ u = \Gamma_u \})\right| = \left| D \Gamma_u(\{ u = \Gamma_u \})\right| \leq 2^d \left| \left\{ u = \Gamma_u \right\}\right|.
\end{equation*}
By~\lref{envelopefacet}, 
\begin{equation*} \label{}
\left|\partial \Gamma_u(U)\right| = \left|\partial \Gamma_u( \{ u = \Gamma_u \})\right|.
\end{equation*}
The last two lines yield~\eqref{e.sharpie}.
\end{proof}

\begin{lemma}
\label{l.mubound}
There exists $c(d,\Lambda)> 0$ such that, for every $F \in \Omega$ and $m \in \Z$,
\begin{equation}\label{e.mubound}
c \inf_{x\in Q_m} (F(0,x))_+^d \leq \mu(Q_m,F) \leq 2^d \sup_{x\in Q_m} (F(0,x))_+^d.
\end{equation}
\end{lemma}

\begin{proof}
The upper bound of~\eqref{e.mubound} follows from \eqref{e.sharpie} after rescaling.

The get the lower bound in~\eqref{e.mubound}, we set $\lambda := \inf_{x\in Q_m} ( F(0,x))_+$ and observe that the parabola $\varphi(x):= (\lambda/2d\Lambda ) |x|^2$ satisfies, for every $x\in Q_m$,
\begin{multline*}\label{}
F(D^2\varphi(x),x) \geq \Pu^-_{1,\Lambda}(D^2\varphi(x)) + F(0,x) = -\Lambda \Delta \varphi(x) + F(0,x) \\ \geq -\lambda + \inf_{y\in Q_m} F(0,y) \geq 0.
\end{multline*}
Hence $\varphi \in S(Q_m,F)$ and, noting that $\varphi = \Gamma_\varphi$, we find that 
\begin{equation*}\label{}
\mu(Q_m,F) \geq \frac{\left| \partial \varphi(Q_m) \right|}{|Q_m|} = \fint_{Q_m} \det D^2\varphi(y)\,dy = \left(\frac{\lambda}{d\Lambda} \right)^d. \qedhere
\end{equation*}
\end{proof}

\begin{lemma}
\label{l.musub}
For every $F \in \Omega$, $m\in \Z$ and $n\in \N$,
\begin{equation}
\label{e.musub}
\mu(Q_{m+n},F) \leq \fint_{Q_{m+n}} \mu(Q_m(x),F) \,dx.
\end{equation}
\end{lemma}

\begin{proof}
Fix $u \in S(Q_{m+n},F)$ and apply \lref{mugood} to get that, for every $x \in Q_{m+n}$, 
\begin{equation*}
|\partial \Gamma_u(Q_{m+n} \cap \partial Q_m(x))| = 0.
\end{equation*}
Thus
\begin{equation*}
|\partial \Gamma_u(Q_{m+n})| = \sum_{\{ Q : Q = Q_m(x) \subseteq Q_{m+n}\}} |\partial \Gamma_u(Q)| = \int_{Q_{m+n}} \frac{|\partial \Gamma_u(Q_m(x))|}{|Q_m|} \,dx.
\end{equation*}
The conclusion \eref{musub} is immediate from this, the definition of~$\mu$ and the observation that, if $\widetilde u$ denotes the restriction of $u$ to $\overline Q_m(x)$, then $\widetilde u \in S(Q_m(x),F)$ and $|\partial \Gamma_{\widetilde u}(Q_m(x))| \geq |\partial \Gamma_u(Q_m(x))|$.
\end{proof}

By Lemma~\ref{l.musub} and stationarity~\eqref{e.stat}, for every $m,n\in \N$,
\begin{equation} \label{e.mono1}
\E\left[ \mu(Q_{m+n},F) \right] \leq \E\left[ \mu(Q_m,F) \right].
\end{equation}
Similarly, Lemma~\ref{l.musub},~\eqref{e.stat} and Jensen's inequality yield, for every $p\geq1$ and $m,n\in \N$,
\begin{equation} \label{e.mono2}
\E\left[ \mu(Q_{m+n},F)^p \right] \leq \E\left[ \mu(Q_m,F)^p \right].
\end{equation}

\begin{lemma}
\label{l.mubalance}
Let $\P$ be a probability measure on $(\Omega,\F)$ satisfying~\eqref{e.Fubb} and~\eqref{e.stat}. Then, for every $m \in \Z$, the map
\begin{equation} \label{e.monot}
s \mapsto \E \bracks{\mu(Q_m,F+s)} \quad\mbox{is nondecreasing.}
\end{equation}
Moreover, there exists $s_0\in [-K_0,K_0]$ such that, for every $s\in\R$,
\begin{equation}
\label{e.sbar}
\left\{
\begin{aligned}
& s > s_0
\implies
\lim_{m \to \infty} \E \bracks{\mu(Q_m,F - s)} \leq \lim_{m \to \infty} \E \bracks{\mu(Q_m,F_* + s)} \\
& s < s_0 
\implies
\lim_{m \to \infty} \E \bracks{\mu(Q_m,F - s)} \geq \lim_{m \to \infty} \E \bracks{\mu(Q_m,F_* + s)}
\end{aligned}
\right.
\end{equation}
\end{lemma}
\begin{proof}
The map~$t\mapsto \E \bracks{ \mu(Q_m,F+t)}$
is nondecreasing by definition. By~\eref{mubound}, we see that $\E \bracks{ \mu(Q_m,F-t)}=0$ for every $t \geq K_0$ and $\E \bracks{ \mu(Q_m,F+t)}\geq c(t-K_0)_+$ for $t>K_0$.  By~\eqref{e.mono1}, we have 
\begin{equation}
\overline{\mu}(F):= \inf_{m\in\N} \E \bracks{ \mu(Q_m,F)} = \lim_{m\to \infty} \E \bracks{ \mu(Q_m,F)}
\end{equation}
and we deduce that $t\mapsto \overline{\mu}(F+t)$ is also nondecreasing and satisfies $\overline{\mu}(F-t) =0$ for $t\geq K_0$ and $\overline{\mu}(F+t) \geq c(t-K_0)_+$. Thus the map 
\begin{equation}
t\mapsto \overline{\mu}(F+t) - \overline{\mu}(F^*-t)
\end{equation}
is also nondecreasing, and it is negative for $t<- K_0$ and positive for $t > K_0$. Therefore there exists $s_0\in [-K_0,K_0]$ such that this map is nonpositive for $t<s_0$ and nonnegative for $t>s_0$. 
This completes the proof. 
\end{proof}

\subsection{The decay of $\mu(Q_m,F)$ for large $m$.}

The following theorem is a quantitative statement concerning the decay of $\mu$. Proving it is the main step in the proof of Theorem~\ref{t.full} and the focus of the next three sections. 

\begin{theorem}
\label{t.mudecay}
Let $\P$ be a probability measure on $(\Omega,\F)$ satisfying~~\eqref{e.Fubb},~\eqref{e.frd} and~\eqref{e.stat}. Then there exists a unique $\bar{s}(\P)\in \R$ and constants $\tau(d,\Lambda) \in (0,1)$ and $C(d,\Lambda) \geq 0$ such that, for every $m\in \N$,
\begin{equation}
\label{e.muvariance}
\E \bracks{ \mu(Q_m,F - \bar{s})^2 + \mu(Q_m,F_*+ \bar{s} )^2} \leq C K_0^{2d} \tau^{m}.
\end{equation}
\end{theorem}

Once we prove Theorem~\ref{t.mudecay}, we use a classical concentration-type argument, using the finite range of dependence assumption a second time, to improve our control over the fluctuations of~$\mu$. The argument is given at the end of Section~\ref{S.mudecay}.

\begin{corollary}
\label{c.mudeviation}
Under the assumptions of Theorem~\ref{t.mudecay}, for every $p< d$, there exist $\alpha(p,d,\Lambda)>0$ and $c(d,\Lambda)> 0$ such that, for every $m \in \N$ and $t\geq1$,
\begin{equation*}
\P\bracks{\mu(Q_{m},F - \bar{s}) \geq   K_0^d 3^{-m\alpha}  t } \leq \exp\parens{-c t 3^{mp} }
\end{equation*}
and
\begin{equation*}
\P\bracks{\mu(Q_{m},F_* + \bar{s}) \geq K_0^d 3^{-m\alpha}  t } \leq \exp\parens{-c t 3^{mp} },
\end{equation*}
where $\bar{s}(\P) \in \R$ is as in \tref{mudecay}. Moreover, there exists $c_0(d,\Lambda)>0$ such that $\alpha(p,d,\Lambda) \geq c_0(d-p)$. 
\end{corollary}

\subsection{Identification of the effective equation}
The constant $\bar{s}(\P)$ in Theorem~\ref{t.mudecay} and Corollary~\ref{c.mudeviation} is nothing other than $\overline F(0)$. In fact, we make this the definition of $\overline F$. To get $\overline F(A)$ for general $A\in \Sy$, we apply Theorem~\ref{t.mudecay} to the pushforward $\P_A$ of $\P$ under the map $A\mapsto F_A$ (defined in~\eqref{e.Ftrans} above):
\begin{equation}\label{e.Fbar}
\overline F(A):= \bar{s}(\P_A), \quad \mbox{where $\bar{s}(\P)\in\R$ is the constant in Theorem~\ref{t.mudecay}.}
\end{equation}

To keep our presentation self-contained, we summarize some basic properties of~$\overline F$. First, we see from Lemma~\ref{l.mubalance} that $\left|\overline F(0)\right| \leq K_0$. Uniform ellipticity is inherited from~\eqref{e.Fue} and the monotonicity of $\bar{s}(\P)$ in $\P$. To see this, fix $A,B\in \Sy$ and define a map $\zeta:\Omega\to \Omega$ by
\begin{equation*}\label{}
\zeta(F)(M,x):= \Pu^-_{1,\Lambda}(A-B) + F(M+B,x).
\end{equation*}
According to~\eqref{e.Fue},
\begin{equation*}\label{}
\zeta(F)(M,x) \leq F(A+M,x) = F_A(M,x)
\end{equation*}
and it follows immediately that $\mu(U,\zeta(F)) \leq \mu(U,F_A)$ and $\mu_*(U,\zeta(F)) \geq \mu_*(U,F_A)$ for all $U\subseteq \Rd$, and hence from~\eqref{e.monot}, we have
\begin{equation*}\label{}
\bar{s}(\zeta_{\#}\P) \leq \bar{s}(\P_A) = \overline F(A).
\end{equation*}
On the other hand, since $\zeta(F) - F_B \equiv \Pu^-_{1,\Lambda}(A-B)$ which is a constant, we have
\begin{equation*}\label{}
\bar{s}(\zeta_{\#}\P) = \Pu_{1,\Lambda}^-(A-B)+\overline F(B).
\end{equation*}
We obtain $\Pu_{1,\Lambda}^-(A-B) \leq \overline F(A) - \overline F(B)$. Therefore $\overline F \in \overline\Omega(\Lambda)$.

By similar arguments, we find that $\overline F$ has properties such as positive homogeneity, convexity/concavity, linearity, oddness, etc, provided that $A\mapsto F(A,x)$ has the same property for every $x\in \Rd$ and $\P$--almost surely. Uncovering qualitative properties of $\overline F$ from \emph{averaged} information about $\P$ is more interesting but much more difficult, and little is currently known (although see the estimate for the effective ellipticity in ~\cite{AS1}).

\section{Strict convexity of quasi-maximizers}
\label{sec.SC}

This section contains only deterministic results, so we fix $F \in \Omega$ throughout.

We begin with an assertion concerning the strict convexity of any convex function $w \in C(Q_0)$ which has a subdifferential map $\partial w$ that is uniformly bounded below on small scales, in the sense that, for some suitable small~$n\in \Z$ ($n\ll 0$) and every $x\in Q_0$, we have $|\partial w(Q_{n}(x))| \geq c |Q_{n}|$. The conclusion is that the graph of $w$ must either curve in all directions at an appropriate rate or else bend extremely rapidly away from a hyperplane.

This result is a quantitative version of an idea that has appeared several times in the regularity theory of the Monge-Amp\`ere equation: see for instance~Caffarelli~\cite{C1} and especially the recent preprint of Mooney~\cite[Lemma~2.2]{M}. This connection can be formally motivated by the fact that, for a convex $\varphi\in C^2$ and $n \ll 0$,
\begin{equation*}\label{}
\det D^2\varphi(x) \approx \fint_{Q_{n}(x)} \det D^2\varphi (y) \, dy = \frac{\left| \partial \varphi(Q_{n}) \right|}{\left|Q_{n}\right|}.
\end{equation*}

\begin{lemma}
\label{l.mooney}
There exist $c(d), h(d) > 0$ such that, for every $0<r<1$, $n\in \Z$ such that $3^{n} \leq cr$ and convex function $w \in C(\overline Q_0)$ satisfying 
\begin{equation}
\label{e.mohyp}
\inf_{Q_0} w = \inf_{Q_n} w = 0 \quad \mbox{and} \quad \inf_{x\in Q_0} \frac{|\partial w(Q_n(x))|}{|Q_n|} \geq 1,
\end{equation}
at least one of the following holds: either
\begin{equation}
\label{e.mogrow}
w \geq h r^{2 - 2/d} \quad \mbox{on } \partial Q_0
\end{equation}
or else there exists $e \in \R^d$ with $|e| = 1$ such that 
\begin{equation}
\label{e.mobend}
w \geq h r^{2-2/d} \quad \mbox{on } \{ x \in \overline Q_0 : |e \cdot x| \geq r \}.
\end{equation}
\end{lemma}

\begin{proof}
We argue by the contrapositive: assuming that both \eref{mogrow} and \eref{mobend} fail for fixed $h,r>0$ and $n\in\Z$, with $3^{-n}r$ sufficiently large depending on $d$, we derive a lower bound on $h$. Throughout the proof, $C$ and $c$ denote positive constants which depend only on $d$ and may differ in each occurrence.

We introduce the closed convex set
\begin{equation*}
S := \{ x \in \overline Q_0 : w(x) \leq h r^{2 - 2/d} \},
\end{equation*}
which has nonempty interior by~\eqref{e.mohyp}. According to John's lemma~\cite{J}, there exists an invertible, orientation-preserving affine map $\phi:\Rd\to \Rd$ such that 
\begin{equation}\label{e.mojohn}
\overline B_1 \subseteq \phi(S) \subseteq \overline B_{d}.
\end{equation}
We may write~$\phi(y) = A(y - x_0)$ for a positive definite matrix~$A \in \Sy$ and~$x_0\in S$. 

\emph{Step 1.} We show that
\begin{equation}
\label{e.moAbound}
\lambda_{\mathrm{max}}(A) \leq  C r^{-1}, \quad  \lambda_{\mathrm{min}}(A) \leq C, \quad \mbox{and} \quad \det A \leq C r^{1 - d}.
\end{equation}
Here $\lambda_{\mathrm{max}}(A)$ and $\lambda_{\mathrm{min}}(A)$ denote the largest and smallest eigenvalues of $A$, respectively. Select $e\in \Rd$ with $|e|=1$ such that $Ae=\lambda_{\mathrm{max}}(A)e$. By the first hypothesis of~\eqref{e.mohyp}, there exists~$x_1\in \overline Q_n \cap S$. Thus
\begin{equation}\label{e.anchor}
 |Ax_0|  = |\phi(x_1) - Ax_1| \leq d + \lambda_{\mathrm{max}}(A)|x_1| \leq C \left(1 + 3^{n}\lambda_{\mathrm{max}}(A)\right).
\end{equation}
Using this, we find
\begin{align*}\label{}
S \subseteq \phi^{-1}(\overline B_d) & = \left\{ x \in \Rd \,:\, \left| A(x-x_0) \right| \leq d \right\}  \\
& \subseteq \left\{ x \in \Rd \,:\, \left| Ax \right| \leq C(1+ 3^{n}\lambda_{\mathrm{max}}(A)) \right\} \\
& \subseteq \left\{ x \in \Rd \,:\, |e\cdot Ax|  \leq C(1+ 3^{n}\lambda_{\mathrm{max}}(A))  \right\} \\
& = \left\{ x \in \Rd \,:\, |e\cdot x|  \leq C(\lambda^{-1}_{\mathrm{max}}(A) +3^{n} )  \right\}.
\end{align*}
If $-n$ sufficiently large that $C3^{n}\leq\frac12r$, then we obtain 
\begin{equation*}\label{}
S \subseteq \left\{ x \in \Rd \,:\, |e\cdot x|  \leq C\lambda^{-1}_{\mathrm{max}}(A) + \tfrac 12r  \right\}.
\end{equation*}
This contradicts the assumed failure of~\eqref{e.mobend} unless $\lambda_{\mathrm{max}}(A) \leq C r^{-1}$, which proves the first estimate of~\eqref{e.moAbound}. 

To prove the second estimate of~\eqref{e.moAbound}, we observe that, due to the assumed failure of \eref{mogrow}, there exists~$x_2 \in \partial Q_0 \cap S$ and we find that
\begin{equation*}
(x_1-x_2) \cdot A(x_1-x_2) = (x_1-x_2)\cdot \left( \phi(x_1) - \phi(x_2) \right) \leq |x_1-x_2| \cdot 2d.
\end{equation*}
Since $|x_1 - x_2| \geq \frac12 - C 3^{-n}  \geq \tfrac{1}{4}$, the normalized vector $y:= (x_1-x_2)/|x_1-x_2|$ satisfies $y \cdot A y \leq 8d \leq C$. Hence $\lambda_{\mathrm{min}}(A) \leq C$. 

Finally, we note that the third estimate of~\eqref{e.moAbound} is a consequence of the first two, since $\det A \leq (\lambda_{\mathrm{max}}(A))^{d-1} \lambda_{\mathrm{min}}(A)$.

\emph{Step 2.} We prove the estimate
\begin{equation}
\label{e.moEbound}
|\partial w(E)| \leq C h^d |E|,
\end{equation}
where we have defined the ellipsoid
\begin{equation*}
E : = \phi^{-1}(B_{1/2}) \subseteq S \subseteq \overline Q_0. 
\end{equation*}
Consider the change of variables $\widetilde w(x) := w(\phi^{-1}(x))$. Observe that 
\begin{equation} \label{e.cngesubs}
\partial w(E) =  \phi\!\left( \partial \widetilde w(B_{1/2}) \right).
\end{equation}
By $w\geq 0$, the first inclusion in \eref{mojohn} and the definition of $S$, we have
\begin{equation*}
0 \leq \widetilde w \leq h r^{2-2/d} \quad \mbox{in } B_1.
\end{equation*}
This implies that
\begin{equation*}
|p| \leq  C h r^{2-2/d} \quad \mbox{for every} \ p \in \partial \widetilde w(B_{1/2}).
\end{equation*}
In particular,
\begin{equation*}
|\partial \widetilde w(B_{1/2})| \leq \left| B_{Chr^{2-2/d}}\right| = C h^d r^{2d - 2}.
\end{equation*}
Using this and \eref{moAbound} and~\eqref{e.cngesubs} we reverse the change of variables to obtain
\begin{multline*}
|\partial w(E)| = \left| \phi\!\left( \partial \widetilde w(B_{1/2}) \right) \right| = (\det A) \left|\partial \widetilde w(B_{1/2})\right| \leq C h^d (\det A)^{-1} = Ch^d|E|.
\end{multline*}

\medskip

{\noindent \em Step 3.} We complete the argument, deriving a lower bound on $h$. Consider the set
\begin{equation*}
\widetilde E_n := \{ x \in E \,:\, Q_n(x) \subseteq E \}.
\end{equation*}
Note that $\overline B_{cr} \subseteq E$ by~\eqref{e.moAbound}. Since $E$ is an ellipsoid, it follows that
\begin{equation*}\label{}
\left| \left\{ x\in E\,:\, \dist(x,\partial E) > cr \right\} \right| \geq \tfrac12 |E|.
\end{equation*}
Therefore, provided that $3^{n}\leq cr$, we have
\begin{equation*}
|\widetilde E_n| \geq \tfrac{1}{2} |E|.
\end{equation*}
Combining this with \eref{moEbound}, we obtain
\begin{equation*}
|\partial w(\widetilde E_n)| \leq |\partial w(E)| \leq C h^d |E| \leq C h^d |\widetilde E_n|.
\end{equation*}
Since $\widetilde E_n \subseteq Q_0$ is a union of level $n$ triadic cubes, the hypothesis \eref{mohyp} gives
\begin{equation*}
|\partial w(\widetilde E_n)| = \int_{\widetilde E_n} \frac{|\partial w(Q_n(x))|}{|Q_n|} \,dx \geq |\widetilde E_n|.
\end{equation*}
Combining the above two strings of inequalities, we obtain $1 \leq C h^d$.
\end{proof}

The exponent $2-2/d$ in \lref{mooney} is sharp, even for a smooth convex function~$w$ satisfying the pointwise bound $\det D^2 w \geq 1$, as we see from the following one-parameter family:
\begin{equation*}
w_r (x_1,\ldots,x_d):= \tfrac{1}{2} r^{2-2/d} x_1^2 + \tfrac{1}{2} r^{-2/d} (x_2^2 + \cdots + x_d^2), \quad r>0.
\end{equation*}

We intend to apply Lemma~\ref{l.mooney} to the convex envelope of a function $u\in S(Q_0,F)$ which nearly achieves the supremum in the definition of $\mu(Q_0,F)$, with the hope of obtaining the first alternative. To this end, we require the following lemma, which will allow us to rule out the second alternative. It roughly states that, if $u\in S(Q_0,F)$ grows quickly away from a hyperplane, then there is a smaller-scale cube $Q_n(x)\subseteq Q_0$ such that $\mu(Q_n(x),F)$ is relatively large.

\begin{lemma}
\label{l.valley}
There exist $c(d)>0$ and $h(d,\Lambda) > 1$ such that, if $0 < r < \tfrac14 h^{-1/2}$, $n\in \Z$ such that $3^{n}\leq cr$ and $u \in S(Q_0,F)$ satisfy, for some $e \in \R^d$ with $|e| = 1$, 
\begin{equation}
\label{e.vagrow}
\inf_{Q_0} u = \inf_{Q_n} u = 0 \quad \mbox{and} \quad u \geq h r^2 \mbox{ on } \{ x \in Q_0 : |e \cdot x| \geq r \},
\end{equation}
then there exists $x_0\in Q_0$ such that 
\begin{equation}
\label{e.vamu}
\mu(Q_n(x_0),F) \geq 2.
\end{equation}
\end{lemma}

\begin{proof}
Throughout, $C$ and $c$ denote positive constants that depend only on $d$.  If $h \geq 8 d \Lambda$, then the quadratic function $\varphi(x) := - \tfrac{1}{8} h (e \cdot x)^2 + |x|^2$ satisfies
\begin{equation*}
\Pu^-_{1,\Lambda}(D^2 \varphi) \geq 0 \quad \mbox{in} \ \Rd.
\end{equation*}
Therefore, the function $\widetilde u := u + \varphi$ belongs to $S(Q_0,F)$. Consider the sets
\begin{equation*}
S_1 := \{ x \in\Rd\,:\, |e \cdot x| < r \mbox{ and } |x|^2  < h r^2 \} \quad \mbox{and} \quad S_2 := 2 S_1.
\end{equation*}
Using $r < \tfrac 14 h^{-1/2}$, we see that
\begin{equation} \label{e.S1S2bnd}
S_1 \subseteq S_2 \subseteq B_{2 h^{1/2} r} \subseteq B_{1/2} \subseteq  Q_0.
\end{equation}
Observe that, since $S_1$ is a subset of a rectangular box which has $d-1$ sides of length $h^{1/2}r$ and one side of length $r$, we have
\begin{equation} \label{e.boxiestuff}
|S_2| \leq C h^{(d-1)/2} r^d.
\end{equation}
We next claim that 
\begin{equation} \label{e.growclm}
\inf_{S_2 \setminus S_1} \widetilde u \geq \tfrac{1}{2} h r^2.
\end{equation}
To see this, take $x\in S_2\setminus S_1$ and consider two alternatives: first, if $|x\cdot e| < r$, then $|x|^2 \geq hr^2$ and so
\begin{equation*}\label{}
\widetilde u(x) = u(x) + \varphi(x) \geq -\frac{1}{8} h(x\cdot e)^2 + |x|^2 \geq \frac78 hr^2,
\end{equation*}
while on the other hand, if $|x\cdot e| \geq r$, then $u(x)>hr^2$ by~\eqref{e.vagrow}, and using the fact that $x\in S_2$, we get 
\begin{equation*}\label{}
\widetilde u(x) \geq hr^2 -\frac18h(x\cdot e)^2 + |x|^2 \geq hr^2 -\tfrac18 h(2r)^2 = \tfrac12 hr^2.
\end{equation*}
This completes the proof of the claim~\eqref{e.growclm}. 

Taking $h>4$ large and $c>0$ small, and using that $3^{n} \leq  cr$, we have that $Q_n \subseteq B_r \subseteq S_1$ and thus, using~\eqref{e.vagrow},
\begin{equation}\label{e.squelch}
\inf_{Q_n} \widetilde u \leq \sup_{Q_n} \varphi \leq \sup_{B_r} \varphi \leq r^2 \leq \tfrac{1}{4} h r^2.
\end{equation}
It follows from~\eqref{e.S1S2bnd},~\eqref{e.growclm} and~\eqref{e.squelch} that, for every $p \in \R^d$ such that $|p| < C h^{1/2} r$, the map $x \mapsto \widetilde u(x) - p \cdot x$ attains its infimum over $S_2$ at a point in $S_1$.  Denoting $\widetilde w := \Gamma_{\widetilde u | S_2}$, we find
\begin{equation}\label{e.tildews}
\left|\partial \widetilde w(S_1)\right| \geq C h^{d/2} r^d.
\end{equation}
Using again that $3^n \leq cr$ and making $c>0$ smaller, if necessary, we have
\begin{equation*}
S_1 \subseteq \{ x \in \Rd\,:\, Q_n(x) \subseteq S_2 \}.
\end{equation*}
Observe that, for every $x\in S_1$,
\begin{equation*}
\mu(Q_n(x),F) \geq |Q_n|^{-1} \left|\partial \widetilde w(Q_n(x))\right|.
\end{equation*}
By combining this with~\eqref{e.boxiestuff}, we obtain
\begin{equation*}
|\partial \widetilde w(S_1)| \leq |S_2| \sup_{x \in S_1} \mu(Q_n(x),F) \leq C h^{(d-1)/2} r^d \sup_{x \in S_1} \mu(Q_n(x),F).
\end{equation*}
An application of~\eqref{e.tildews} and a rearrangement yields
\begin{equation*}
h^{1/2} \leq C \sup_{x \in S_1} \mu(Q_n(x),F).
\end{equation*}
Taking $h\geq C$, we obtain \eref{vamu}.
\end{proof}

The next result states that, if the value of $\mu$ on every small scale cube is close to its value on the large scale cube, then the graph of a function $u$ which (nearly) attains the supremum in the definition of $\mu$ for the large scale cube must have curvature in all directions: after subtracting off a plane, it must look like a bowl.

\begin{lemma}
\label{l.convex}
There is a constant $c(d,\Lambda) > 0$ such that, if $n \leq n_0(d,\Lambda) < 0$ and $u \in S(Q_1,F)$ satisfies
\begin{equation}
\label{e.coquasi}
1 \leq \frac{|\partial \Gamma_u (Q_n(x))|}{|Q_n|} \leq \mu(Q_n(x),F) \leq 1 + 3^{dn} \quad \mbox{for } x \in Q_0,
\end{equation}
then there is a point $x_0 \in \{ \Gamma_u = u \} \cap Q_n$ and a slope $p_0 \in \partial \Gamma_u(x_0)$ such that
\begin{equation}
\label{e.cogrow}
u(x) \geq u(x_0) + p_0 \cdot (x - x_0) + c \quad \mbox{for } x \in Q_1 \setminus Q_0.
\end{equation}
\end{lemma}

\begin{proof}
The idea is to apply \lref{mooney} to $\Gamma_u$ and then use \lref{valley} to rule out the second alternative~\eref{mobend} in the conclusion of Lemma~\ref{l.mooney}.

We may assume that, for some $x_0 \in Q_n$,
\begin{equation}
\label{e.coflat}
u(x_0) = \Gamma_u(x_0) = 0 \quad \mbox{and} \quad 0 \in \partial \Gamma_u(x_0),
\end{equation}
Indeed, by \eref{coquasi} we have, for every $y\in Q_0$,
\begin{equation*}
\begin{aligned}
1 & \leq \int_{Q_0} \frac{|\partial \Gamma_u(Q_n(x))|}{|Q_n|} \,dx \\
& = |\partial \Gamma_u(Q_n(y))| + \int_{Q_0 \setminus Q_n(y)} \frac{|\partial \Gamma_u(Q_n(x))|}{|Q_n|} \,dx \\
& \leq |\partial \Gamma_u(Q_n(y))| + (1 - 3^{dn})(1 + 3^{dn}).
\end{aligned}
\end{equation*}
In particular, for any $y\in Q_0$, we have $|\partial \Gamma_u(Q_n(y))| > 0$ and hence
\begin{equation}
\label{e.cocontact}
Q_n(y) \cap \{ u=\Gamma_u \} \neq \emptyset.
\end{equation}
Now, we choose $x_0 \in Q_n \cap \{ u = \Gamma_u \}$ and $p_0 \in \partial \Gamma_u(x_0)$, and subtract the affine function $x \mapsto u(x_0) + p_0 \cdot (x - x_0)$ from both $u$ and $\Gamma_u$.  This gives $u(x_0) = \Gamma_u(x_0) = 0$ and $0 \in \partial \Gamma_u(x_0)$ while preserving the hypotheses of the lemma.

Take $r>0$ to be selected below. Applying \lref{mooney}, we find that, provided $n \leq n_0(d,r) < 0$, either~\eref{mogrow} or~\eref{mobend} holds for $\Gamma_u$. In the case in which~\eref{mobend} holds and $r$ is sufficiently small, depending on $(d, \Lambda)$,  \lref{valley} gives
\begin{equation*}
\mu(Q_n(x_1),F) \geq 2
\end{equation*}
for some $x_1 \in Q_0$, contradicting \eref{coquasi}. Thus the first alternative~\eqref{e.mogrow} must hold and, in view of \eref{coflat}, we obtain
\begin{equation*}
\inf_{Q_1 \setminus Q_0} u \geq \inf_{\partial Q_0} \Gamma_u \geq \inf_{Q_n} \Gamma_u + h r^{2-2/d} = \inf_{Q_n} u + h r^{2 - 2/d},
\end{equation*}
where $h(d)>0$ is as in~\lref{mooney}. In particular, \eref{cogrow} holds for $p_0 = 0$ and $c = h r^{2 - 2/d}>0$.
\end{proof}

We next rescale Lemma~\ref{l.convex} to get a statement which is better suited for its main application (which is found in Step~2 of the proof of Lemma~\ref{l.contract} below). 

\begin{corollary}
\label{c.convex}
There is a constant $c(d,\Lambda) > 0$ such that, if $n \geq n_0(d,\Lambda) > 0$, $m \in \Z$, $a > 0$, and $u \in S(Q_{m+n+1},F)$ satisfy
\begin{equation}
\label{e.cquasi}
a \leq \frac{|\partial \Gamma_u(Q_m(x))|}{|Q_m|} \leq \mu(Q_m(x),F) \leq \left(1 + 3^{-dn}\right) a \quad \mbox{for all}  \ x \in Q_{m+n},
\end{equation}
then there exists $x_0 \in \{\Gamma_u = u\} \cap Q_m$ and $p_0 \in \partial \Gamma_u(x_0)$ such that
\begin{equation}
\label{e.cgrow}
u(x) \geq u(x_0) + p_0 \cdot (x - x_0) + c a^{1/d} \left(3^{m+n}\right)^2 \quad \mbox{for all} \ x \in Q_{m+n+1} \setminus Q_{m+n}.
\end{equation}
\end{corollary}

\begin{proof}
For every $s,t>0$ and $F\in \Omega$, the operator $G:\Sy\times \Rd\to\R$ defined by
\begin{equation*}
G(A,x) := t^{-1} F(t A, s x)
\end{equation*}
belongs to $\Omega$. Moreover, if $u \in S(U,F)$, then the function $v(x) := t^{-1} s^{-2} v(s x)$ belongs to $S(s^{-1}U,G)$.
Since the constants in \lref{convex} depend only $(d,\Lambda)$, we immediately obtain \eref{cgrow} from \eref{cquasi} by taking $s := 3^{m+n}$ and $t := a^{1/d}$ and applying \lref{convex} with $G$ and $v$ in place of $F$ and $u$.
\end{proof}

\section{Contraction of the Variance}

In this section we establish the two key ingredients in the proof of Theorem~\ref{t.mudecay}. They are (i) Lemma~\ref{l.contract}, which is based on the theory in the previous section and asserts that, if the variances of $\mu$ and $\mu_*$ are both small (relative to their second moments) then, on a larger scale, both $\mu$ and $\mu_*$ have small second moments; and (ii) Lemma~\ref{l.concentrate}, which uses the finite range of dependence to show that, after passing to a larger scale, the second moment of $\mu$ must decay no less than by an amount proportional to its variance. 

Throughout this section, we assume that $\P$ is a probability measure on~$(\Omega(\Lambda),\F)$ satisfying~\eqref{e.Fubb},~\eqref{e.stat} and~\eqref{e.frd}.

\begin{lemma}
\label{l.contract}
Suppose $s,\delta > 0$ and $m,n\in \N$ are such that
\begin{equation}\label{e.contcth1}
0 < \E \bracks{ \mu(Q_m,F + s)^2 } \leq (1 + \delta) \E \bracks{ \mu(Q_{m+n},F + s) }^2
\end{equation}
and
\begin{equation}\label{e.contcth2}
0 < \E \bracks{ \mu(Q_m,F_* + s)^2 } \leq (1 + \delta) \E \bracks{ \mu(Q_{m+n},F_* + s) }^2.
\end{equation}
Then there exist $C(d,\Lambda) >0$, $n_0(d,\Lambda)\in \N$ and $\delta_0(d,\Lambda)>0$ such that $n\geq n_0$ and $\delta \leq \delta_0$ imply that 
\begin{equation}
\label{e.contract}
\E \bracks{ \mu(Q_{m+n},F + s)^2 } + \E \bracks{ \mu(Q_{m+n},F_* + s)^2 } 
\leq 
Cs^{2d}.
\end{equation}
\end{lemma}

\begin{proof}
By scaling, we may assume that $m = 0$.  Define
\begin{equation*}
a := \E \bracks{\mu(Q_n,F+s)} \quad \mbox{and} \quad a_* := \E \bracks{\mu(Q_n,F_*+s)} = \E \bracks{\mu_*(Q_n,F-s)} .
\end{equation*}
Also fix $\ep > 0$ to be selected below. Throughout the proof, we let $C$ and $c$ denote positive constants that depends only on $(d,\Lambda)$ and may differ in each occurrence. 

\smallskip

\emph{Step 1.} We show that, if $\delta < 3^{-dn-1} \ep^2$, then there exists $F \in \Omega$ such that, for all $x\in Q_n$,
\begin{equation}
\label{e.cquasi1}
(1 - \ep) a \leq \mu(Q_n,F+s) 
\quad \mbox{and} \quad 
 \mu(Q_0(x),F+s) \leq (1 + \ep) a
\end{equation}
and
\begin{equation}
\label{e.cquasi2}
(1 - \ep) a_* \leq \mu(Q_n,F_*+s) 
\quad \mbox{and} \quad 
\mu(Q_0(x),F_*+s) \leq (1 + \ep) a_*.
\end{equation}
Using Chebyshev's inequality,~\eqref{e.mono1},~\eqref{e.contcth1} and $a=\E\bracks{\mu(Q_n,F+s)}$, we estimate
\begin{equation*}
\begin{aligned}
\P\bracks{\mu(Q_0,F+s) > (1 + \ep) a} & \leq \P\bracks{(\mu(Q_0,F+s)-a)^2 > \ep^2 a^2} \\
& \leq \frac1{\ep^2 a^2} \E\bracks{(\mu(Q_0,F+s)-a)^2} \\
& \leq \frac1{\ep^2 a^2} \left( \E\bracks{(\mu(Q_0,F+s)^2}-a^2 \right)\\
& \leq \delta \ep^{-2}.
\end{aligned} 
\end{equation*}
Using Chebyshev's inequality,~\eqref{e.mono2} and~\eqref{e.contcth1}, we compute
\begin{equation*}
\begin{aligned}
\P\bracks{\mu(Q_n,F+s) < (1 - \ep)a} & \leq \P\bracks{(\mu(Q_n,F+s)-a)^2 > \ep^2 a^2} \\
& \leq \frac{\E\bracks{\mu(Q_n,F+s)^2} - a^2}{\ep^2 a^2} \\
& \leq \frac{\E\bracks{\mu(Q_0,F+s)^2} - a^2}{\ep^2 a^2} \\
& \leq \delta \ep^{-2}.
\end{aligned}
\end{equation*}
Using~\eqref{e.contcth2} in place of~\eqref{e.contcth1} and arguing similarly, we obtain
\begin{equation*}
\P\bracks{\mu(Q_0,F_*+s) > (1 + \ep) a_*} \leq \delta \ep^{-2}
\end{equation*}
and
\begin{equation*}
\P\bracks{\mu(Q_n,F_*+s) < (1 - \ep)a_*} \leq \delta \ep^{-2}.
\end{equation*}
The  above four inequalities and a union bound tell us that the probability that both \eref{cquasi1} and \eref{cquasi2} hold is at least
\begin{equation*}
1 - 2(3^{dn} + 1) \delta \ep^{-2} \geq 1 - 3^{dn + 1}  \delta \ep^{-2}.
\end{equation*}
If $\delta < 3^{-dn-1}\ep^2$, then this probability is positive and  in particular there exists $F\in \Omega$ for which both~\eqref{e.cquasi1} and~\eqref{e.cquasi2} hold.

\smallskip

\emph{Step 2.} We show that, if $\ep < 3^{-2dn-2}$ and $F\in \Omega$ is such that both \eref{cquasi1} and \eref{cquasi2} hold, 
then 
\begin{equation}
\label{e.ca3dn}
a + a_* \geq c3^{dn}(a + a_* - C s^d).
\end{equation}
We begin by observing that there exist $u,u_*\in C(\overline Q_n)$ satisfying 
\begin{equation}\label{e.equust}
F(D^2u,x) + s = 0 = F_*(D^2u_*,x) +s  \quad \mbox{in} \ Q_n,
\end{equation}
\begin{equation}
\label{e.cbump1}
\inf_{\partial Q_n} u \geq \inf_{Q_n} u + c 3^{2n}a^{1/d}   \quad \mbox{and} \quad \inf_{Q_n} u = \inf_{Q_0} u = 0.\end{equation}
and
\begin{equation}
\label{e.cbump2}
\inf_{\partial Q_n} u_* \geq \inf_{Q_n} u_* + c  3^{2n} a_*^{1/d} \quad \mbox{and} \quad \inf_{Q_n} u_* = \inf_{Q_0} u_* = 0.
\end{equation}
Indeed, we first select~$u \in S(Q_n,F+s)$ such that
\begin{equation}
\label{e.select}
\frac{|\partial \Gamma_u(Q_n)|}{|Q_n|} \geq (1 - \ep) \mu(Q_n,F+s)
\end{equation}
and then check that \eqref{e.cquasi1} implies that the hypothesis of \cref{convex} holds for $u$, using that~$\ep < 3^{-2dn-2}$. To see this, use~\eqref{e.cquasi1} and~\eqref{e.select} to get
\begin{align}
\label{e.wrap}
(1-\ep)^2 a  
\leq 
\frac{|\partial \Gamma_u(Q_n)|}{|Q_n|} 
&
\leq
\frac1{|Q_n|} 
\sum_{Q_0(x) \subseteq Q_n} |\partial \Gamma_u(Q_0(x))|
\\ & \notag
\leq 
\frac1{|Q_n|} 
\sum_{Q_0(x) \subseteq Q_n}
\mu(Q_0(x),F+s)
\leq (1+\ep)a. 
\end{align}
The first inequality is by~\eqref{e.cquasi1} and~\eqref{e.select} above, the second is from the fact that the cubes~$Q_0(x)\subseteq Q_n$ partition~$Q_n$ (and the fact that $| \partial \Gamma_u(K)|=0$ if $|K|=0$, see Lemma 2.2), the third inequality is by definition, and the last one is by~\eqref{e.cquasi1} and the fact that the sum is over exactly $3^{dn}=|Q_n|$ many cubes. The string of inequalities is thus strict by no more than $(1+\ep) a - (1-\ep)^2a \leq 3\ep a$. We deduce that, for every $x\in Q_n$, 
\begin{equation}
|\partial \Gamma_u(Q_0(x))| \geq (1+\ep)a-3\ep a|Q_n| \geq \big( 1 - 3^{1-dn} \ep \big) a. 
\end{equation}
Combined with~\eqref{e.cquasi1} and $\ep < 3^{-2dn-2}$, this implies the hypothesis of Corollary~\ref{c.convex} is valid for $u$ with~$m=0$ and~$ \big( 1 - 3^{1-dn} \ep \big) a$ in place of~$a$. Now, after applying the corollary to get~\eqref{e.cgrow} and subtracting an affine function from $u$, we obtain \eref{cbump1}. By replacing $u$ by the solution $\widetilde u$ of $F(D^2\tilde u) + s =0$ in $Q_n$ with Dirichlet boundary condition $\widetilde u = u$ on $\partial Q_n$, we may assume that the first equation of~\eqref{e.equust} holds (we have also used that, by the comparison principle, $\Gamma_u(Q_n) \subseteq \Gamma_{\widetilde u}(Q_n)$). The same argument also works to produce $u_*$. Observe that it is here, in the application of Corollary~\ref{c.convex}, that we have used the hypothesis that $a>0$ and $a_* > 0$. 

\smallskip

Using~\eqref{e.odd.dual}, we see that the function $w:=u+u_*$ satisfies\footnote{For readers who may not be experts in viscosity solution technicalities: of course, the differential inequality for $w$ is formally derived from~\eqref{e.Fue}, but it is not immediately obvious that this is rigorous in the viscosity sense because it is possible that neither $u$ nor $u_*$ is $C^2$. It turns out that the inequality is valid, but must be justified by an argument based on the comparison principle, which goes like this: if $w$ is not a supersolution of the inequality, then by definition we can strictly touch it from below by a smooth function $\phi$ which violates the inequality. Then we compare $u$ to $u_*-\phi$ to get a contradiction. This argument is well-known and so we omit the details, and we make free use of this technical device throughout the paper without further mention.}
\begin{equation*}
w \geq c3^{2n}\left(a^{1/d}+a_*^{1/d}\right) \quad \mbox{on } \partial Q_n \quad \mbox{and} \quad \Pu^+_{1,\Lambda}(D^2w) \geq -2s  \quad \mbox{in} \ Q_n.
\end{equation*}
By comparing to a parabola (or alternatively, using the ABP inequality), we obtain
\begin{equation}
\label{e.uusabove}
w \geq c3^{2n}\left(a^{1/d}+a_*^{1/d}\right) - C 3^{2n}s \quad \mbox{in } Q_n.
\end{equation}
Now let $v,v_* \in C(\overline Q_1)$ denote the solutions of
\begin{equation*}
\begin{cases}
F(D^2 v,x) + s = 0 & \mbox{in } Q_1, \\
v = 0 & \mbox{on } \partial Q_1,
\end{cases}
\quad \mbox{and} \quad
\begin{cases}
F_*(D^2 v_*,x) + s = 0 & \mbox{in } Q_1, \\
v_* = 0 & \mbox{on } \partial Q_1,
\end{cases}
\end{equation*}
and observe that their sum $\widetilde w := v + v_*$ satisfies
\begin{equation*}
\widetilde w = 0 \quad \mbox{on } \partial Q_1 \quad \mbox{and} \quad \Pu^-_{1,\Lambda}(D^2 \widetilde w) \leq -2s \leq 0 \quad \mbox{in} \ Q_1.
\end{equation*}
By the maximum principle,
\begin{equation}
\label{e.vvsbelow}
\widetilde w \leq 0 \quad \mbox{in} \ Q_1.
\end{equation}
Combining \eref{cbump1}, \eref{cbump2}, \eref{uusabove}, and \eref{vvsbelow}, we have
\begin{equation} \label{e.fandango}
v(0) - u(0) + v_*(0) - u_*(0) = \widetilde w(0) - w(0) \leq C 3^{2n}s - c3^{2n}\left(a^{1/d}+a_*^{1/d}\right).
\end{equation}
Thus at least one of the terms $v(0) - u(0)$ or $v_*(0) - u_*(0)$ is no more than half the right side of~\eqref{e.fandango}, i.e., no more than $C3^{2n}s$. By symmetry, we may assume without loss of generality that
\begin{equation} \label{e.updembo}
v(0) - u(0) \leq C 3^{2n}s - c3^{2n}\left(a^{1/d}+a_*^{1/d}\right)
\end{equation}
and consider the difference $\xi := v-u$, which satisfies
\begin{equation*}
\xi \leq 0 \quad \mbox{on} \ \partial Q_1 \quad \mbox{and} \quad \Pu^-_{1,\Lambda}(D^2\xi) \leq 0 \leq \Pu^+_{1,\Lambda}(D^2\xi) \quad \mbox{in} \ Q_1.
\end{equation*}
The maximum principle gives that $\xi \leq 0$ in $Q_1$ and, in view of~\eqref{e.updembo}, the Harnack inequality~\cite[Theorem~4.3]{CC} implies
\begin{equation*}
v - u = \xi \leq c\xi(0) \leq - c3^{2n} (a^{1/d} + a_*^{1/d} - C s)  \quad \mbox{in} \ Q_0.
\end{equation*}
Therefore,
\begin{equation*}
\inf_{Q_0} v \leq \inf_{Q_0} u - c3^{2n}(a^{1/d} + a_*^{1/d} - C s) = - c3^{2n}(a^{1/d} + a_*^{1/d} - C s).
\end{equation*}
Using this and the fact that $v = 0$ on $\partial Q_1$, we may apply Lemma~\ref{l.muabp} to get
\begin{equation*}\label{}
 c3^{dn}(a^{1/d} + a_*^{1/d} - C s)^d  \leq \mu(Q_1,F) \leq \fint_{Q_1} \mu(Q_0(x),F)\,dx \leq (1 + \ep)a.
\end{equation*}
Note that in the last inequality we used~\eqref{e.cquasi1}. Since 
\begin{equation*}\label{}
(a^{1/d} + a_*^{1/d} - C s)^d \geq c(a+a_*-Cs^d),
\end{equation*}
this completes the proof of~\eqref{e.ca3dn}.

\smallskip

\emph{Step 3.} The conclusion. By Steps~1 and~2, if $\delta < 3^{-5dn-5}$, then 
\begin{equation*} \label{}
\left( c3^{dn} - 1\right) ( a + a_* ) \leq C3^{dn} s^d. 
\end{equation*}
Taking $n_0(d,\Lambda)\in\N$ sufficiently large, we deduce that, if~$n \geq n_0$, then 
\begin{equation*}
a + a_*
\leq
C s^d. 
\end{equation*}
From the previous inequality and~\eqref{e.mono2},~\eqref{e.contcth1} and~\eqref{e.contcth2}, we obtain
\begin{align*}
\E \bracks{ \mu(Q_{n},F + s)^2 } + \E \bracks{ \mu(Q_{n},F_* + s)^2 }  &  \leq  \E \bracks{ \mu(Q_{0},F + s)^2 } + \E \bracks{ \mu(Q_{0},F_* + s)^2 } \\
& \leq (1+\delta)\left( a^2 + a_*^2 \right) \\
& \leq C s^{2d}.
\end{align*}
This proves~\eqref{e.contract} for $n = n_0$ and $\delta < \delta_0:= 3^{-5dn_0-5}$. By~\eqref{e.mono1}, the hypotheses of the lemma are stronger and the conclusion is weaker as~$n$ becomes larger. We deduce therefore that the lemma is valid for every $n\geq n_0$ and $\delta < \delta_0$.
\end{proof}

The next lemma contains the only use of the unit range of dependence assumption in the proof of Theorem~\ref{t.mudecay}. In preparation, we observe that it is immediate from the definitions that, for every bounded convex domain $U \subseteq \Rd$,
\begin{equation}\label{e.mumeas}
F \mapsto \mu(U,F) \quad \mbox{is $\F(U)$--measurable.}
\end{equation}

\begin{lemma}
\label{l.concentrate}
There is a constant $C(d) > 0$ such that, for all $n, m\in \N$,
\begin{equation}
\label{e.concentrate}
\E \bracks{ \mu(Q_{m+n},F)^2 } \leq  \E \bracks{ \mu(Q_m,F) }^2 + C 3^{-nd/2} \E \bracks{ \mu(Q_m,F)^2 }.
\end{equation}
\end{lemma}

\begin{proof}
For every $m,n\in\N$, and $\delta > 0$,
\begin{align}\label{e.young}
& \lefteqn{\mu(Q_{m+n},F)^2  \leq \left( \fint_{Q_{m+n}} \mu(Q_m(x),F)\,dx \right)^2}  \\
& \qquad = \left( \fint_{Q_{m+n}} \left( \mu(Q_m(x),F) - \E\left[\mu(Q_m,F) \right] \right) \,dx \right)^2 + \E\left[\mu(Q_m,F) \right]^2 \nonumber \\ 
& \qquad \qquad +2 \E\left[\mu(Q_m,F) \right] \left( \fint_{Q_{m+n}} \left( \mu(Q_m(x),F) - \E\left[\mu(Q_m,F) \right] \right) \,dx \right)  \nonumber \\
& \qquad\leq \left( 1 + \frac1\delta\right) \left( \fint_{Q_{m+n}} \left( \mu(Q_m(x),F) - \E\left[\mu(Q_m,F) \right] \right) \,dx \right)^2 \nonumber  \\
&  \qquad \qquad+  (1+\delta)  \E\left[\mu(Q_m,F) \right]^2, \nonumber 
\end{align}
where the last line was obtained by Young's inequality. To estimate the expectation of the first term on the last line, we observe that
\begin{align}\label{e.varcov}
& \E \bracks{\parens{ \fint_{Q_{m+n}} \left( \mu(Q_m(x),F) - \E\bracks{\mu(Q_m,F)}\right) \,dx }^2} \\
& \qquad = \E\left[ \Bigg( 3^{-dn} \sum_{1 \leq i \leq 3^{dn}}  \left( \mu(Q_m(x_i),F) - \E\bracks{\mu(Q_m,F)}\right) \Bigg)^2\,\right] \nonumber \\
& \qquad = 3^{-2dn} \sum_{1 \leq i, j \leq 3^{dn}} \cov\left[\mu(Q_m(x_i),F)\,;\mu(Q_m(x_{j}),F)\right],\nonumber
\end{align}
where $\left\{ Q_m(x_i)\,:\, i = 1,\ldots 3^{dn}\right\}$ is an enumeration of the subcubes of $Q_{m+n}$ of the form $Q_m(x)$.  Due to~\eqref{e.frd},~\eqref{e.mumeas} and $m\geq 0$, we see that
\begin{equation}\label{e.usingfrd}
\cov\left[ \mu(Q_m(x),F) ; \mu(Q_m(y),F) \right]  = 0
\end{equation}
unless $\dist(Q_m(x),Q_m(y)) =0$, and so at most $3^{d(n+1)}\leq C3^{dn}$ terms in the sum on the last line of~\eqref{e.varcov} are nonzero. These we bound by H\"older's inequality and~\eqref{e.stat}:
\begin{equation*}\label{}
\left| \cov\left[ \mu(Q_m(x),F) ; \mu(Q_m(y),F) \right] \right| \leq \var\left[ \mu(Q_m,F) \right].
\end{equation*}
Using the previous line and~\eqref{e.varcov}, we estimate the expectation of~\eqref{e.young} by
\begin{equation*}
\E \bracks{\mu(Q_{m+n}(x),F)^2} \leq C (1 + \delta^{-1}) 3^{-dn} \var\bracks{\mu(Q_m,F)} + (1 + \delta) \E\bracks{\mu(Q_m,F)}^2.
\end{equation*}
Taking $\delta := 3^{-nd/2}$ and rearranging this expression yields the lemma.
\end{proof}

\section{Decay of $\mu$: the proof of \tref{mudecay}}
\label{S.mudecay}

In this section we present the proofs of Theorem~\ref{t.mudecay} and Corollary~\ref{c.mudeviation}. Throughout we assume that $\P$ satisfies~\eqref{e.Fubb},~\eqref{e.stat} and~\eqref{e.frd}. 

\smallskip

We begin by showing that, if $\E[\mu]$ and $\E[\mu_*]$ are balanced in the large-scale limit, then $\E[\mu]$ becomes strictly positive after adding a positive constant to~$F$.

\begin{lemma}
\label{l.pushup}
Suppose that, for every $s>0$,
\begin{equation} \label{e.bonk}
\lim_{n\to \infty} \E \left[ \mu(Q_n,F+s) \right] \geq \lim_{n\to \infty} \E \left[ \mu(Q_n,F_*-s) \right].
\end{equation}
Then there exists $c(d,\Lambda)>0$ such that, for every $m \in \Z$ and $s > 0$,
\begin{equation*}
\E \left[ \mu(Q_m,F + s) \right] \geq c s^d.
\end{equation*}
\end{lemma}
\begin{proof}
Throughout, $C$ and $c$ denote positive constants depending on $(d,\Lambda)$ which may differ in each occurrence. Set
\begin{equation*}
a := \sup_{s>0} \lim_{n \to \infty} \E \left[ \mu(Q_n,F_*-s)\right]  \geq 0.
\end{equation*}
Fix $s>0$, $m\in \Z$, $\delta > 0$ and select $M \geq m$ such that
\begin{equation*}
\E \left[\mu(Q_M,F_*-s)\right] \leq  a + \delta.
\end{equation*}
Let $v_*(\cdot,F) \in C(Q_M)$ denote the solution of
\begin{equation*}
\left\{ \begin{aligned}
&F_*(D^2 v_*,x) = s & \mbox{in} & \ Q_M, \\
& v_* = 0 & \mbox{on} & \ \partial Q_M.
\end{aligned} \right.
\end{equation*}
By Lemma 2.1 and Chebyshev's inequality,
\begin{equation} \label{e.chebbsie}
\P\left[ \inf_{Q_M} v_* \leq - C (a + \delta)^{1/d} 3^{2M} \right] \leq\P\left[  \mu(Q_M,F_*-s) \geq 2 (a+\delta) \right] \leq \frac{1}{2}.
\end{equation}
Next, observe that for $c(d,\Lambda)>0$, the function
\begin{equation*}
v(x,F) := -c s \left(\tfrac{1}{4} \cdot3^{2M} - |x|^2\right) - v_*(x,F)
\end{equation*}
satisfies $v(\cdot,F) \in S(Q_M,F+s)$. According to~\eqref{e.chebbsie}, we find that
\begin{equation*}
\P\left[ \inf_{Q_M} v \leq C ((a + \delta)^{1/d} - c s) 3^{2M} \right]
\geq \frac{1}{2}.
\end{equation*}
Using this and the fact that $v \geq 0$ on $\partial Q_M$ and applying Lemma~2.1, we find that 
\begin{equation*}
\P \left[ \mu(Q_M, F+s) \geq  c s^d - C(a+\delta)\right] \geq \tfrac{1}{2},
\end{equation*}
By (2.11), we find that
\begin{equation*}
\E \left[ \mu(Q_m,F+s)\right] \geq \E \left[ \mu(Q_M,F+s) \right] \geq c s^d - C(a + \delta).
\end{equation*}
We also have, by the assumption~\eqref{e.bonk} and the fact that 
\begin{equation*} \label{}
s\mapsto \E \left[ \mu(Q_n,F+s) \right] \quad \mbox{is nondecreasing} 
\end{equation*}
that
\begin{equation*}
\E \left[ \mu(Q_m,F+s) \right] \geq \sup_{s'>0} \lim_{n \to \infty} \E \left[ \mu(Q_n,F-s') \right]= a.
\end{equation*}
We conclude by sending $\delta \to 0$ and observing that $\max\{a, cs^d - Ca \} \geq c s^d$.
\end{proof}

We now give the proof of Theorem~\ref{t.mudecay}.

\begin{proof}[{\bf Proof of \tref{mudecay}}]
According to Lemma~\ref{l.mubalance}, by subtracting a constant from~$F$, we may assume that, for every~$s>0$,
\begin{equation} \label{e.assump1}
\lim_{n\to \infty} \E \left[ \mu(Q_n,F+s) \right] \geq \lim_{n\to \infty} \E \left[ \mu(Q_n,F_*-s) \right]
\end{equation}
and 
\begin{equation} \label{e.assump2}
\lim_{n\to \infty} \E \left[ \mu(Q_n,F-s) \right] \leq \lim_{n\to \infty} \E \left[ \mu(Q_n,F_*+s) \right].
\end{equation}
Under this assumption, we will prove that~\eqref{e.muvariance} holds with $\overline{s}=0$, that is, for some constant~$\tau(d,\Lambda)<1$, 
\begin{equation}
\label{e.muvariance2}
\E \bracks{ \mu(Q_m,F)^2 + \mu(Q_m,F_* )^2} \leq C K_0^{2d} \tau^{m}.
\end{equation}
The estimate~\eqref{e.muvariance2} confirms the existence of $\bar{s}$ as in the statement of the theorem. 
The uniqueness of~$\bar{s}=\bar{s}(\P)$ is then an immediately consequence of Lemma~\ref{l.pushup}.

\smallskip

The proof of~\eqref{e.muvariance2} is broken into four steps. As usual, $C$ and $c$ denote positive constants depending only on $(d,\Lambda)$ which may differ in each instance.

\smallskip

\emph{Step 1.} For each $m,k\in \N$, we define the quantities
\begin{equation*}
\begin{aligned}
a(m,k) & := \E\bracks{ \mu(Q_m,F+2^{-k}) }^2, & b(m,k) & := \E\bracks{ \mu(Q_m,F+2^{-k})^2 }, \\
a_*(m,k) & := \E\bracks{ \mu(Q_m,F_*+2^{-k}) }^2, & b_*(m,k) & := \E\bracks{ \mu(Q_m,F_*+2^{-k})^2 }.
\end{aligned}
\end{equation*}
According to~\eqref{e.mono1},~\eqref{e.mono2} and~\eqref{e.monot}, each of these quantites is nonincreasing in both variables $m$ and $k$. By Lemma~\ref{l.pushup} and the assumption that~\eqref{e.assump1} and~\eqref{e.assump2} hold for every $s>0$, we have, for every $m,k\in\N$,
\begin{equation*}
c 2^{-2dk} \leq a(m,k) \leq b(m,k)
\end{equation*}
and 
\begin{equation*}
c 2^{-2dk} \leq a_*(m,k) \leq b_*(m,k).
\end{equation*}
Fix $n_1\in \N$ and $\delta_1>0$ to be selected below. 

\smallskip

\emph{Step 2.}
We claim that there exists $m\in \N$ such that 
\begin{equation} \label{e.mbnds1}
n_1 \leq m \leq n_1 + \frac{4n_1}{\delta_1} \log(C 2^{2dk} (b(0,k) + b_*(0,k)))
\end{equation}
and
\begin{equation*}
\begin{aligned}
a(m - n_1, k) & \leq (1 + \delta_1) a(m,k), & b(m - n_1, k) & \leq (1 + \delta_1) b(m,k), \\
a_*(m - n_1, k) & \leq (1 + \delta_1) a_*(m,k), & b_*(m - n_1, k) & \leq (1 + \delta_1) b_*(m,k).
\end{aligned}
\end{equation*}
To see this, we use the estimates from Step 1 to obtain that, for every $M\in \N$,
\begin{equation*}
\begin{aligned}
\prod_{j = 1}^{4M+1} \frac{a((j-1) n_1,k)}{a(j n_1,k)} & \leq C 2^{2dk} b(0,k) & \prod_{j = 1}^{4M+1} \frac{b((j-1) n_1,k)}{b(j n_1,k)} & \leq C 2^{2dk} b(0,k) \\
\prod_{j = 1}^{4M+1} \frac{a_*((j-1) n_1,k)}{a_*(j n_1,k)} & \leq C 2^{2dk} b_*(0,k) & \prod_{j = 1}^{4M+1} \frac{b_*((j-1) n_1,k)}{b_*(j n_1,k)} & \leq C 2^{2dk} b_*(0,k). \\
\end{aligned}
\end{equation*}
Here is some more detail on the derivation of the first inequality (the other three are obtained similarly):
\begin{equation*} \label{}
\prod_{j = 1}^{4M+1} \frac{a((j-1) n_1,k)}{a(j n_1,k)}  = \frac{a(4Mn_1,k)}{a(n_1,k)} \leq C2^{2dk}b(4Mn_1,k) \leq C2^{2dk}b(0,k).
\end{equation*}
Since each factor in these products is at least $1$, by the monotonicity of the four quantities in the first variable, it follows that, for some $1\leq j \leq 4M+1$,
\begin{equation*}
\begin{aligned}
\frac{a((j-1) n_1,k)}{a(j n_1),k)} & \leq (C 2^{2dk} b(0,k))^{1/M}, & \frac{b((j-1) n_1,k)}{b(j n_1),k)} & \leq (C 2^{2dk} b(0,k))^{1/M}, \\
\frac{a_*((j-1) n_1,k)}{a_*(j n_1),k)} & \leq (C 2^{2dk} b_*(0,k))^{1/M}, & \frac{b_*((j-1) n_1,k)}{b_*(j n_1),k)} & \leq (C 2^{2dk} b_*(0,k))^{1/M}. \\
\end{aligned}
\end{equation*}
We conclude the proof of the claim by taking $m := j n_1$ and setting
\begin{equation*}
M := \left\lceil \frac{\log(C 2^{2dk} (b(0,k)+b_*(0,k)))}{ \log(1 + \delta_1)} \right\rceil.
\end{equation*}
Here $\lceil r \rceil$ denotes, for $r\in \R$, the smallest integer not smaller than $r$.

\smallskip

\emph{Step 3.} We show that
\begin{equation}
\label{e.bstep}
b(m,k) + b_*(m,k) \leq C2^{-2dk}.
\end{equation}
Let $n_0\in \N$ and $\delta_0 > 0$ be the constants from the statement of \lref{contract} and assume $n_1>n_0$. We first apply \lref{concentrate} to get
\begin{equation} \label{e.blaggardy}
b(m-n_0,k) \leq C 3^{-(n_1-n_0) d/2} b(m-n_1,k) + a(m - n_1,k).
\end{equation}
By Step 2, we have
\begin{equation*}
b(m-n_1,k) \leq (1 + \delta_1) b(m,k) \leq (1 + \delta_1) b(m-n_0,k)
\end{equation*}
and
\begin{equation*}
a(m-n_1,k) \leq (1 + \delta_1) a(m,k).
\end{equation*}
Substituting these into~\eqref{e.blaggardy} and rearranging, we obtain
\begin{equation*}
b(m-n_0,k) \leq C 3^{-(n_1-n_0)d/2} (1 + \delta_1)b(m-n_0,k) + (1+\delta_1) a(m,k).
\end{equation*}
Now select $0<\delta_1 (d,\Lambda)\leq\frac12$ such that $(1+\delta_1) (1-\delta_1)^{-1} \leq 1+\delta_0$ and then take $n_1(d,\Lambda)$ large enough that $C3^{-(n_1-n_0)d/2} < \delta_1$ to obtain
\begin{equation*}
b(m-n_0,k) \leq (1 + \delta_0) a(m,k).
\end{equation*}
By an identical argument, we also obtain
\begin{equation*}
b_*(m-n_0,k) \leq (1 + \delta_0) a_*(m,k).
\end{equation*}
Now an application of Lemma~\ref{l.contract} yields~\eqref{e.bstep}. Observe that $n_1$ may be chosen so that $n_0 < n_1 \leq n_0 + C$. Therefore, by~\eqref{e.mbnds1} we have
\begin{equation*} \label{e.mbnds2}
n_0 \leq m \leq n_0 + C \log\left(C2^{2dk} (b(0,k) + b_*(0,k))\right).
\end{equation*}

\emph{Step 4.} We complete the proof by iterating Step 3. We define $\{ m_k \}_{k=0}^\infty \subseteq \N$ inductively as follows. Take $m_0 := 0$ and, given $m_k$, let $m_{k+1}$ be least integer $m$ larger than $m_{k}$ such that~\eqref{e.bstep} holds. According to Step 3, we have
\begin{equation*}
m_{k+1} - m_k \leq C \log \left(C2^{2dk}(b(m_k,k) + b_*(m_k,k))\right).
\end{equation*}
Since
\begin{equation*}
b(m_{k},k) + b_*(m_{k},k) \leq b(m_{k},k-1) + b_*(m_{k},k-1) \leq C2^{-2d(k-1)} \leq C2^{-2dk},
\end{equation*}
we obtain, for every $k\in \N$ with $k\geq 1$,
\begin{equation*}
m_{k+1} \leq  m_k + C.
\end{equation*}
Using \lref{mubound} to estimate the first step, we have
\begin{equation*}
m_1 \leq C \log\left( CK_0^{2d}\right).
\end{equation*}
Finally, we apply~\eqref{e.monot} to obtain
\begin{equation*}
\E \bracks{ \mu(Q_{m_k},F)^2 + \mu(Q_{m_k},F_*)^2 } \leq b(m_k,k) + b_*(m_k,k) \leq C2^{-2dk}.
\end{equation*}
Using the monotonicity of $s\mapsto \E\left[ \mu(Q,F+s)^2\right]$ to interpolate for $m$'s in between successive $m_k$'s, we obtain \eref{muvariance}.
\end{proof}

\begin{proof}[{\bf Proof of \cref{mudeviation}}]
Let $\bar{s}=\bar{s}(\P)$ be as in Theorem~\ref{t.mudecay}. We may suppose without loss of generality that $\bar{s}=0$.

\smallskip

We adapt the classical concentration argument as in for example the proofs of Bernstein's inequalities. Let $\{ Q^j_{n+1} : 1 \leq j \leq 3^{dm} \}$ be an enumeration of the subcubes of $Q_{m+n+1}$ of the form $Q_{n+1}(x)$. Next, for each $1\leq j \leq 3^{dm}$, we let $\{ Q_n^{j,i} \,:\, 1 \leq i \leq 3^d\}$ be an enumeration of the subcubes of $Q_{n+1}^j$ of the form $Q_n(x)$, such that, for every $1\leq j,j'\leq 3^{dm}$ and $1\leq i \leq 3^d$, the translation which maps $Q_{n+1}^j$ onto $Q_{n+1}^{j'}$ also maps $Q_{n}^{j,i}$ onto $Q_{n}^{j',i}$. In particular, for every $1 \leq i \leq 3^d$ and $1 \leq j < j' \leq 3^{dm}$, we have $\dist(Q_n^{i,j},Q_n^{i,j'}) \geq 1$ and therefore, by~\eqref{e.frd} and~\eqref{e.mumeas},
\begin{equation}
\label{e.expdist}
F \mapsto \mu(Q_{n}^{j,i},F) \quad \mbox{and} \quad F \mapsto \mu(Q_{n}^{j',i},F) \quad \mbox{are independent.}
\end{equation}
Using this enumeration of subcubes, we compute
\begin{align*}
& \ \hspace{-2em} \log \E \bracks{\exp\parens{t3^{dm}  \mu(Q_{m+n+1},F)}} \\
& \leq \log \E \bracks{ \prod_{1 \leq i \leq 3^d} \prod_{1 \leq j \leq 3^{dm}} \exp\parens{t3^{-d} \mu(Q_n^{j,i},F)}} & \mbox{(by \eqref{e.musub})}\\
& \leq 3^{-d} \sum_{1 \leq i \leq 3^d} \log  \E \bracks{ \prod_{1 \leq j \leq 3^{dm}} \exp\parens{t\mu(Q_n^{j,i},F)}} & \mbox{(H\"older ineq.)} \\
& = 3^{-d} \sum_{1 \leq i \leq 3^d} \log  \prod_{1 \leq j \leq 3^{dm}} \E \bracks{  \exp\parens{t\mu(Q_n^{j,i},F)}} & \mbox{(by \eqref{e.expdist})} \\
& = 3^{dm} \log \E \bracks{  \exp\parens{t\mu(Q_n,F)}}. & \mbox{(by~\eqref{e.stat})}
\end{align*}   
Take $t:= 1/(2K_0^d)$ and estimating the last term using the elementary inequalities
\begin{equation*}\label{}
\left\{ \begin{aligned}
& \exp(s) \leq 1+2s && \mbox{for all} \ 0\leq s \leq 1, \\
& \log(1+s) \leq s && \mbox{for all} \ s\geq 0,
\end{aligned} \right.
\end{equation*}
and the fact that $\P\left[ \mu(Q_n,F)  \leq (2 K_0)^d \right] =1$ by~\eqref{e.Fubb} and~Lemma~\ref{l.mubound}, to obtain
\begin{align*}  
\log \E \bracks{\exp\parens{ 3^{dm} (2K_0^d)^{-1}\mu(Q_{m+n+1},F)}} \leq  2 \cdot 3^{dm} \E\left[ (2 K_0)^{-d}\mu(Q_n,F) \right].
\end{align*}
Theorem~\ref{t.mudecay} yields
\begin{equation*}\label{}
\log \E \bracks{\exp\parens{3^{dm} (2 K_0)^{-d} \mu(Q_{m+n+1},F)}} \leq C 3^{dm} \tau^n.
\end{equation*}
Finally, an application of Chebyshev's inequality gives
\begin{equation*} \label{}
 \P\bracks{ \mu(Q_{m+n},F) \geq t K_0^d } \leq \exp\left( -3^{dm} \left( t - C \tau^n \right) \right).
\end{equation*}
Replacing $t$ with $C \tau^nt$, we get
\begin{equation*} \label{}
 \P\bracks{ \mu(Q_{m+n},F) \geq t K_0^d \tau^n } \leq \exp\left( -c t3^{dm}  \tau^n  \right).
\end{equation*}
We obtain the first assertion of the corollary from this expression by choosing 
\begin{equation*} \label{}
n:= \left\lfloor \frac{(d-p)m}{p+a} \right\rfloor \quad \mbox{and} \quad \alpha:= \frac{a(d-p)}{d+a}, \quad \mbox{where} \quad a:= \frac{|\log \tau|}{\log 3}
\end{equation*}
and replacing $m+n$ by $m$. A symmetric argument yields the same estimate for $F_*$ in place of $F$.
\end{proof}

\section{The proof of Theorem~\ref{t.full}}
\label{S.full}

In this section we use the decay of $\mu$ to control the difference $\sup_{x\in U} |u^\ep-u|$ between solutions of the Dirichlet problem for the heterogeneous and homogeneous problems, enabling us to deduce Theorem~\ref{t.full} from Theorem~\ref{t.mudecay}. The argument is entirely deterministic and the precise statement is given in Proposition~\ref{p.snapgrid} below, which states that, if $\sup_{x\in U} (u-u^\ep)$ is relatively large, then we can find a matrix $A^*\in\Sy$ with $\overline F(A^*) \leq 0$, where $A^*$ is chosen from a preselected finite list, and a large cube $Q^*\subseteq U$, also chosen from a preselected finite list of such cubes, such that $\mu(Q^*,F_{A^*})$ is also relatively large. Recall that $F_A\in \Omega$ is defined in~\eqref{e.Ftrans}.

If the homogenized limit function $u$ is $C^2$, then the idea is fairly straightforward: near a point where $u-u^\ep$ has a local maximum, we may essentially replace $u$ by a quadratic function. The Hessian of this quadratic function is $A^*$, and we use Lemma~\ref{l.muabp} with the difference of $u^\ep$ and the quadratic function as the witness, to conclude that $\mu(Q^*,F_{A^*})$ must be relatively large in some (rescaled, large) cube~$Q^*$. A technical difficulty arises because solutions of uniformly elliptic equations are not in general $C^2$. To resolve this issue, we rely on the regularity theory, in particular the $W^{2,\sigma}$ and $W^{3,\sigma}$ estimates (here $\sigma > 0$ is tiny, see~\cite{CC} and~\cite[Lemma 5.2]{ASS}) which give quadratic expansions for solutions of constant-coefficient equations in sets of large measure. This is essentially the same idea as the one used by Caffarelli and Souganidis in Sections~5 and~6 of~\cite{CS}. 

We begin with a simple ``double--variable" variation of Lemma~\ref{l.mugood}. It gives a lower bound for the Lebesgue measure in $\Rd\times\Rd$ of the set of points at which we can touch the difference of a subsolution $u$ and supersolution $v$ by planes, after doubling the variables and adding the usual quadratic penalization term.

\begin{lemma}
\label{l.smear}
Let $U \subseteq \Rd$ be open, $K \geq 0$ and $u,v\in C(\overline U)$ satisfy
\begin{equation*}\label{e.smeareq}
\Pu^-_{1,\Lambda} (D^2u) -K \leq 0 \leq \Pu^+_{1,\Lambda} (D^2v) + K \quad \mbox{in} \ U.
\end{equation*}
Assume~$\delta > 0$, $V=\overline V\subseteq U\times U$ and  $W \subseteq \Rd\times\Rd$ such that, for every $(p,q) \in W$,
\begin{multline*}\label{}
\sup_{(x,y) \in V} \left( u(x) - v(y) - \tfrac1{2\delta} |x-y|^2 - p\cdot x - q\cdot y \right) \\ =\sup_{(x,y)\in U\times U} \left( u(x) - v(y) - \tfrac1{2\delta} |x-y|^2 - p\cdot x - q\cdot y \right).
\end{multline*}
Then there exists $C=C(d,\Lambda)>0$ such that 
\begin{equation*} \label{}
\left| W \right| \leq \left(2K+C\delta^{-1} \right)^{2d} \left| V \right|.
\end{equation*}
\end{lemma}
\begin{proof}
As usual, $C>0$ denotes a positive constant depending on $(d,\Lambda)$ which may differ in each occurrence. It suffices to show that, for every pair $(x_i,y_i,p_i,q_i) \in U\times U \times \Rd \times \Rd$, $i=1,2$, such that
\begin{multline} \label{e.doubleyourass}
u(x_i) - v(y_i) - \tfrac1{2\delta} |x_i-y_i|^2 - p_i\cdot x_i - q_i\cdot y_i \\ =\sup_{(x,y)\in U\times U} \left( u(x) - v(y) - \tfrac1{2\delta} |x-y|^2 - p_i\cdot x - q_i\cdot y \right),
\end{multline}
and $|x_1-x_2|^2 + |y_1-y_2|^2 \leq r^2$, we have
\begin{equation} \label{e.doubleLIP}
\left( \left| p_1-p_2 \right|^2 + |q_1-q_2|^2 \right)^{1/2} \leq \left( 2K+C/\delta \right) r + o(r) \quad \mbox{as} \ r \to 0.
\end{equation}
Indeed, from~\eqref{e.doubleLIP} the conclusion follows at once from elementary properties of Lebesgue measure.

We first observe that, by Lemma~\ref{l.envelopec11}, if $s:= |x_1-x_2| < \tfrac12\dist(x_1,\partial U)$, then
\begin{equation} \label{e.slooopies}
\partial \Gamma_{\widetilde u} (B_s(x_1)) \subseteq B_{(2K+C/\delta)(s+Cs^3)}(-p_1),
\end{equation}
where we have defined
\begin{equation*} \label{}
\widetilde u(x):= -u(x) + \frac{1}{2\delta}|x-y_1|^2.
\end{equation*}
Indeed, we just need to check the hypotheses of the lemma. It is clear that $\widetilde u$ satisfies
\begin{equation*} \label{}
\Pu^+_{1,\Lambda}(D^2\widetilde u) \geq -\left( K + d\Lambda \delta^{-1}\right)  \geq - \left( K + C\delta^{-1}\right) \quad \mbox{in} \ U.
\end{equation*}
According to ~\eqref{e.doubleyourass} with $i=1$, we have $\widetilde u(x_1) = \Gamma_{\widetilde u}(x_1)$ and $-p_1 \in \partial \Gamma_{\widetilde u}(x_1)$. Thus Lemma~\ref{l.envelopec11} gives~\eqref{e.slooopies}.

We next check that
\begin{equation} \label{e.gooochies}
-p_2 + \frac{y_2-y_1}{\delta} \in \partial \Gamma_{\widetilde u}(x_2).
\end{equation}
In fact, this follows immediately from~\eqref{e.doubleyourass} with $i=2$, since the latter implies
\begin{equation*} \label{}
x\mapsto u(x) - \frac1{2\delta}|x-y_2|^2 -p_2\cdot x\quad \mbox{achieves its supremum over $U$ at} \ x_2,
\end{equation*}
and
\begin{equation*} \label{}
u(x) - \frac1{2\delta}|x-y_2|^2 -p_2\cdot x = -\widetilde u(x) - \left( p_2 - \frac{y_2-y_1}{\delta}  \right) \cdot x + \frac1{2\delta}\left( |y_2|^2 - |y_1|^2 \right).
\end{equation*}
According to~\eqref{e.slooopies} and~\eqref{e.gooochies},
\begin{equation*} \label{}
\left| p_1 - p_2 + \frac{y_2-y_1}{\delta} \right| \leq (2K+C/\delta)\left(|x_1-x_2|+C|x_1-x_2|^3\right).
\end{equation*}
Rearranging, we obtain
\begin{equation*} \label{}
\left| p_1 - p_2 \right| \leq (2K+C/\delta)(|x_1-x_2|+C|x_1-x_2|^3) + \tfrac{1}{\delta} |y_1-y_2|.
\end{equation*}
By symmetry, we also get
\begin{equation*} \label{}
\left| q_1 - q_2 \right| \leq (2K+C/\delta)(|y_1-y_2|+C|y_1-y_2|^3)+ \tfrac{1}{\delta} |x_1-x_2|
\end{equation*}
and combining the last two lines yields~\eqref{e.doubleLIP}. This completes the proof.
\end{proof}

The next proposition is the deterministic link between Theorems~\ref{t.full} and~\ref{t.mudecay}. Its proof is based on the comparison principle, quantified by the $W^{2,\sigma}$ and $W^{3,\sigma}$ estimates (these can be essentially found in~\cite{CC,ASS}; see also Remark~\ref{r.outsource} below).

\begin{proposition}
\label{p.snapgrid}
Suppose $U \subseteq \R^d$ is a smooth bounded domain and the functions $u, v \in C(\overline U)$ satisfy
\begin{equation*}
\begin{cases}
G(D^2u) = f = F\left(D^2v,x\right) & \mbox{in } U \\
u = g = v & \mbox{on } \partial U,
\end{cases}
\end{equation*}
where $G\in \overline \Omega(\Lambda)$, $F \in \Omega(\Lambda)$, $g \in C^{0,1}(\partial U)$, and $f\in C^{0,1}(U)$ satisfy
\begin{equation*}
|G(0)| + \sup_{x\in U} |F(0,x)| + \| g \|_{C^{0,1}(\partial U)} + \| f \|_{C^{0,1}(U)} \leq K_0 < + \infty.
\end{equation*}
There is an exponent $\kappa \in (0,1)$ depending only on $d$ and $\Lambda$ and constants $C, c > 0$ depending only on $d$,  $\Lambda$, and $U$ such that, for all $0<l\leq h$ such that
\begin{equation} \label{e.bigerr}
E := \sup_{x \in U} (u - v)(x) \geq C K_0 h^\kappa > 0,
\end{equation}
there exist $A^* \in \Sy$ and $y^* \in U$ which satisfy the following:
\begin{itemize}
\item $|A^*| \leq h^{\kappa-1}$,
\item $l^{-1} A^*$ and $h^{-1} y^*$ have integer entries, and
\item $\mu(Q^*,F_{A^*} - G(A^*)) \geq c E^{d}$, where $Q^*:=y^*+ 2 h Q_0.$ 
\end{itemize}
\end{proposition}

\begin{proof}
Throughout the proof, $C$ and $c$ denote positive constants which depend only on $d$, $\Lambda$, and $U$ but may be different in each instance.

{\em Step 1.}  We make several initial observations. First, we may assume without loss of generality that $U \subseteq B_1$ and $K_0=1$, using the same rescaling/normalizing argument as in \cref{convex}. Second, by comparing $v$ to the function $x\mapsto v(x) +\frac12 E(1-|x|^2)$ (or alternatively, using the ABP inequality), we may replace the equation for $v$ by 
\begin{equation}\label{e.ptfmstrict}
F(D^2v) = f + c E \quad \mbox{in} \ B_1. 
\end{equation}
Indeed, otherwise we replace $E$ by $\tfrac12E$ and $v$ by the solution of the Dirichlet problem for~\eqref{e.ptfmstrict} with the same boundary condition. Third, in view of the bound $K_0 \leq 1$ and the smoothness of $U$, the global H\"older estimates yield, for $\sigma(d,\Lambda) \in (0,1)$,
\begin{equation*}
\| u \|_{C^\sigma(\bar U)} + \| v \|_{C^\sigma(\bar U)} \leq C.
\end{equation*}
Since $u = v$ on $\partial U$, the triangle inequality gives, for every $x, y \in U$,
\begin{equation}\label{e.ptfmlips}
|u(x) - v(y)| \leq C \dist(\{x,y\},\partial U)^\sigma + C |x - y|^\sigma.
\end{equation}
For convenience we may take $0<\sigma \leq \tfrac12$.

{\em Step 2.}
We use Lemma~\ref{l.smear} to find a relatively large set on which $v$ touches $u$ from above, after tilting and translating the functions.

We consider the auxiliary function $\Phi:\overline U \times \overline U \times \R^d \times \R^d \to \R$ defined by
\begin{equation*}\label{}
\Phi(x,y,p,q):= u(x) - v(y) - \frac1{2\delta} |x-y|^2 - p\cdot x - q\cdot y,
\end{equation*} 
for some $\delta > 0$ to be determined.  Choose $x_0 \in U$ such that $\Phi(x_0,x_0,0,0) = E$. Set $r := \min \braces{ \tfrac{1}{8} E, 1 }$. Given $p, q \in B_r$, we compute
\begin{equation*}
\Phi(x_0,x_0,p,q) \geq \tfrac{3}{4} E
\end{equation*}
and estimate
\begin{equation*}
\begin{aligned}
\Phi(x,y,p,q) & = u(x) - v(y) - \frac{1}{2\delta}|x - y|^2 - p \cdot x -q\cdot y \\
& \leq C \dist(\{ x, y\},\partial U)^\sigma + C |x-y|^\sigma - \frac{1}{2 \delta} |x - y|^2 + 2r \\
& \leq C \dist(\{ x, y \}, \partial U)^\sigma + \tfrac{1}{4} E + \parens{C E^{-(2 - \sigma)/\sigma} - \frac{1}{2\delta}} |x - y|^2 + \tfrac14E\\
& \leq \tfrac{1}{2} E + C \dist( \{ x, y \}, \partial U)^\sigma,
\end{aligned}
\end{equation*}
where in the third line it was Young's inequality that gave us
\begin{equation*}\label{}
|x-y|^{\sigma} = E^{(2-\sigma)/2} \left( E^{-(2-\sigma)/\sigma}|x-y|^2 \right)^{\sigma/2} \leq \tfrac14E + CE^{-(2-\sigma)/\sigma}|x-y|^2
\end{equation*}
and to get the fourth line of the inequality string, we must impose the condition 
\begin{equation*}
\delta \leq c E^{(2-\sigma)/\sigma}.
\end{equation*}
Then we may fix $\delta:= c E^{(2-\sigma)/\sigma}$ so that, for all $p, q \in B_r$, the map $(x,y) \mapsto \Phi(x,y,p,q)$ attains its supremum in $U \times U$ on $\overline U_s \times \overline U_s$, where $s := c E^{1/\sigma}$. Here we have denoted
\begin{equation*}
U_s := \{ x \in U \,:\, \dist(x,\partial U) > s \}.
\end{equation*}
Let $Z$ be the set of points where such supremums are attained:
\begin{equation*}\label{}
Z:= \left\{ (x,y) \in U_s \times U_s\,:\, \exists \, (p,q) \in B_r \times B_r, \ \Phi(x,y,p,q) = \sup_{ U \times U} \Phi(\cdot,p,q) \right\}
\end{equation*}
and apply Lemma~\ref{l.smear} to conclude that
\begin{equation*}\label{}
|Z| \geq c \delta^{2d} r^{2d} \geq c\left( E^{(2-\sigma)/\sigma}\right)^{2d} E^{2d} = c E^{4d/\sigma}.
\end{equation*}
Let $\pi_1:\R^{d}\times \Rd \to \R^d$ be the projection onto the first $d$ variables, i.e., $\pi_1(x,y): = x$ for every $x,y\in\Rd$. Then we obtain
\begin{equation}\label{e.pluggies}
\left| \pi_1(Z) \right| \geq |U_s|^{-1} |Z| \geq |B_1|^{-1} |Z| \geq c E^{4d/\sigma}. 
\end{equation}
Finally, we note that, for every $(x,y) \in Z$, we can see from $\Phi(x,y,p,q) \geq 0$ for some $p,q\in B_{1}$  and $\sigma \leq \tfrac12$ that 
\begin{equation} \label{e.ptfmxy}
|x-y|^2 \leq C \delta \leq C E^{(2-\sigma)/\sigma} \leq CE^3. 
\end{equation}
 
\emph{Step 3.} We show that there are points $(x,y)\in Z$ such that $u$ has an appropriate quadratic expansion at $x$. Let $P_t$ be the set of points at which $u$ has a global quadratic expansion with both a quadratic term of size $t>0$ and a cubic error term of size $t>0$:
\begin{multline*}\label{}
P_t := \Big\{ x\in U \, : \, \exists \, (A,\xi) \in \Sy \times \Rd \ \mbox{such that} \ |A| \leq t \ \mbox{and, for all} \ z\in U, \\ 
 \left| u(z) - u(x) - \xi\cdot (z-x) - \tfrac12 (z-x) \cdot A(z-x) \right| \leq \tfrac16 t |z-x|^3 \Big\}.
\end{multline*}
According to the $W^{2,\sigma}$ and $W^{3,\sigma}$ estimates (see Remark~\ref{r.outsource} below), we have
\begin{equation*}\label{}
\left| U
 \setminus P_t  \right| \leq C t^{-\sigma},
\end{equation*}
where the exponent $\sigma > 0$ depends only on $d$ and $\Lambda$ (we may reuse the symbol $\sigma$ by taking the minimum of this $\sigma$ with the one from Step~1). In view of~\eqref{e.pluggies}, we have, for every $t\geq CE^{-4d/\sigma^2}$,
\begin{equation*}\label{}
\left| U \setminus P_t  \right| <  \left| \pi_1(Z) \right|.
\end{equation*}
We henceforth take $t\geq CE^{-4d/\sigma^2}$ to be a fixed constant, which will be selected below. In particular, we have $\pi_1(Z) \cap P_t \neq \emptyset$.

\emph{Step 4.}
We complete the proof by exhibiting $A^*$,~$y^*$ and~$Q^*$ as in the conclusion of the proposition. By the previous step, there exists~$(x_1,y_1) \in Z$ with~$x_1\in P_t$. Select~$p,q \in B_r$ such that
\begin{equation}\label{e.wrut1}
\Phi(x_1,y_1,p,q) = \sup_{x,y \in U}  \Phi(x,y,p,q)
\end{equation}
and $(A,\xi) \in \Sy\times \Rd$ such that $|A| \leq t$ and, for all $z\in U$,
\begin{equation}\label{e.wrut2}
\left| u(z) - u(x_1) - \xi\cdot (z-x_1) - \tfrac12 (z-x_1) \cdot A(z-x_1) \right| \leq \tfrac16t|z-x_1|^3.
\end{equation}
Note that $G(A) = f(x_1)$, since $u$ satisfies $G(D^2u) = f$ in $U$ and $u$ is touched from above and below at $x_1$ by cubic polynomials with Hessians equal to $A$ at $x_1$. Combining~\eqref{e.wrut1} and~\eqref{e.wrut2} gives 
\begin{multline}\label{e.dbltch}
\phi(x_1) - v(y_1) - \tfrac1{2\delta}|x_1-y_1|^2 - q\cdot y_1 \\ = \sup_{x,y\in U} \left( \phi(x) - v(y) - \tfrac1{2\delta}|x-y|^2 - q\cdot y \right).
\end{multline}
where $\phi$ is the cubic polynomial defined by
\begin{equation*}\label{}
\phi(z) := u(x_1) + (\xi-p) \cdot(z-x_1) + \tfrac12(z-x_1) \cdot A(z-x_1) - \tfrac16 t |z-x_1|^3.
\end{equation*}
Observe that, for each $y\in U$, we have
\begin{align*} 
\sup_{x \in U} \left( \phi(x) - \tfrac1{2\delta}|x-y|^2 \right) 
& \geq \phi(x_1+(y-y_1)) - \tfrac1{2\delta} |x_1-y_1|^2\\
& = \phi(x_1) - \tfrac1{2\delta}|x_1-y_1|^2+ (\xi-p) \cdot(y-y_1) \\ & \qquad + \tfrac12(y-y_1) \cdot A(y-y_1) - \tfrac16 t |y-y_1|^3.
\end{align*}
Inserting this into~\eqref{e.dbltch}, using $u(x_1)=\phi(x_1)$ and rearranging, we obtain
\begin{equation*}\label{}
v(y_1)  =  \inf_{y\in U} \Big( v(y) - (\xi-p-q) \cdot(y-y_1)- \tfrac12(y-y_1) \cdot A(y-y_1) + \tfrac16 t |y-y_1|^3 \Big).
\end{equation*}

Since $l \leq h$, we may select $A^*\in \Sy$ satisfying $A \leq A^* \leq A + Ch^\kappa I_d$ such that $l^{-1} A^*$ has integer entries. By ellipticity,~$G(A^*) \leq G(A)=f(x_1)$. Define
\begin{equation*}\label{}
w(y):= v(y) - (\xi-p-q) \cdot(y-y_1) - \tfrac12(y-y_1) \cdot (A-c_0E I_d)(y-y_1) + \tfrac16 t |y-y_1|^3,
\end{equation*}
with $c_0>0$ to be selected. In view of~\eqref{e.ptfmstrict}, we check that $w$ satisfies
\begin{align}\label{e.ptfmpert}
F\left(A^*+D^2w,x\right) & \geq F(D^2v,x) - Cc_0 E - Ct|x-y_1| \\ & \geq f(y_1) + cE -Cc_0E- C(t+1)|x-y_1|. \nonumber
\end{align}
The first inequality of~\eqref{e.ptfmpert} is a priori merely formal, but as usual we can obtain this in the viscosity sense (even more easily this time, since~$w$ is a smooth perturbation of~$v$). Taking $c_0$ small enough, we obtain
\begin{equation}\label{e.ptfmpert2}
F\left(A^*+D^2w,x\right)  \geq f(y_1)  \quad  \mbox{in} \ B_{cE/(t+1)}(y_1).
\end{equation}
Moreover, we have
\begin{equation}\label{e.ptfmbmp}
w(y_1) = \inf_{y\in U} \left(  w - cE |y-y_1|^2\right).
\end{equation}
We now select $y^*$ so that $h^{-1} y^* \in \Z^d$ and $|y_1-y^*| \leq \sqrt d h$.

We next check that, for appropriate choices of $E$ and $t$, we have 
\begin{equation}\label{e.inclQ}
Q^*:=y^* + 2 hQ_0 \subseteq U \cap B_{cE/t}(y_1).
\end{equation}
First, we note that $Q^* \subseteq U$ provided that $h\geq CE^{1/\sigma}$, since $y^*\in U_s$. For the second inclusion, we need to choose the parameters so that $cE/t \geq (1 + \sqrt{d})h$. We may satisfy this condition, as well as the requirement imposed in Step~2 that~$t\geq CE^{-4d/\sigma^2}$, by choosing~$t:=h^{\kappa-1}$ where~$\kappa:= (1+4d/\sigma^2)^{-1}$. Then all is well, provided that $E\geq Ch^\kappa$, as assumed in~\eqref{e.bigerr}. Moreover, using~\eqref{e.ptfmxy}, $E \geq C h^\kappa$ and $|y_1-y^*| \leq Ch \ll E$, we deduce that that the right side of~\eqref{e.ptfmpert2} is larger than $f(x_1) \geq G(A^*)$. In particular, $w\in \mathcal{S}(Q^*,F_{A^*} - G(A^*))$. Using this,~\eqref{e.ptfmbmp} and~\eqref{e.inclQ}, an application of Lemma~\ref{l.muabp} yields
\begin{align*}\label{}
\mu(Q^*,F_{A^*} - G(A^*))^{1/d}
& \geq 
c h^{-2} \left( \inf_{\partial Q^*} w - w(y_1) \right) 
\geq 
cE.
\end{align*}
This completes the proof.
\end{proof}

\begin{remark}
\label{r.outsource}
In the proof of Proposition~6.2 above, we used the $W^{2,\sigma}$ and $W^{3,\sigma}$ estimates for a solution of $F(D^2u) = f$, with $f$ Lipschitz. These estimates are essentially contained in~\cite{CC}, and more precise statements we need can be found for example in \cite{ASS}. However, the estimates from~\cite{ASS} (Proposition~3.1 and Lemma~5.2 of that paper) are stated terms of solutions of $\Pu^+_{1,\Lambda}(D^2u) \geq 0$ and $F(D^2u) = 0$, so the hypotheses do not quite fit. 

Here is why the arguments of~\cite{ASS} generalize without any difficulty to our case, giving us what we need: 

\begin{itemize}

\item By replacing $u$ by the sum of $u$ and a parabola in the statement of \cite[Proposition 3.1]{ASS}, the $W^{2,\sigma}$ estimates can be easily formulated in terms of solutions of $\Pu^+_{1,\Lambda}(D^2u) \geq -1$. This applies in particular to solutions of $F(D^2u) = f$, with $f$ bounded. 

\item In the proof of~\cite[Lemma~5.2]{ASS}, one differentiates the equation $F(D^2u) = 0$ to obtain that, for any unit vector $e\in\partial B_1$, the function $v:=\partial_e u$ satisfies
\begin{equation*}\label{}
\Pu^-_{1,\Lambda}(D^2v ) \leq 0 \leq \Pu^-_{1,\Lambda}(D^2v).
\end{equation*}
If instead $u$ solves $F(D^2u) =f$ with $f$ Lipschitz, the same calculation gives  
\begin{equation*}\label{}
\Pu^-_{1,\Lambda}(D^2v) -K \leq 0 \leq \Pu^-_{1,\Lambda}(D^2v) +K,
\end{equation*}
where $K$ is the Lipschitz constant of $f$. The proof then proceeds as before, using the form of the $W^{2,\sigma}$ from the first step. 
\end{itemize}

\end{remark}

We now present the final piece of the argument of the main result. What remains is to combine Proposition~\ref{p.snapgrid} and Theorem~\ref{t.mudecay}, which is fairly straightforward but involves juggling some constants and careful bookkeeping. 

\begin{proof}[{\bf Proof of Theorem~\ref{t.full}}]
Fix $p\in (0,d)$. Denote $q:= (p+2d)/3$ and $q':=(2p+d)/3$ so that $p< q'<q<d$, and take $\alpha (q,d,\Lambda)$ to be as in the statement of Corollary~\ref{c.mudeviation} and $\kappa(d,\Lambda)$ to be the exponent in Proposition~\ref{p.snapgrid}. By scaling (as in the proof of Corollary~\ref{c.convex}), we may assume without loss of generality that $U \subseteq B_1$ and 
\begin{equation*} \label{}
K_0 + \| g \|_{C^{0,1}(\partial U)} + \| f \|_{C^{0,1}(U)} \leq 1.
\end{equation*}
As usual, $C$ and $c$ denote positive constants which depend on $d$, $\Lambda$, $U$ and $p$ and may differ in each occurrence.

\smallskip

We present only the proof that, for some $\beta(p,d,\Lambda)>0$,
\begin{equation}\label{e.wtf}
\P\left[ \sup_{x\in U} \left( u(x) - u^\ep(x,F) \right) \geq C \ep^{\beta} \right] \leq C\exp\left(-\ep^{-p} \right),
\end{equation}
that is, the lower bound for~$u^\ep - u$. The proof of the upper bound for~$u^\ep-u$ is then immediately obtained by applying this result to the pushforward of $\P$ under the map $F\mapsto F_*$ (or by replacing $F$ by $F_*$ and repeating the argument). The fact that $\beta(p,d,\Lambda) \geq c(d,\Lambda)(d-p)$ will be implicit in the argument. 

\smallskip

Fix $\ep\in(0,1)$. Let $m\in\N$ be the smallest positive integer such that
\begin{equation*} \label{}
\max\left\{ 3^{-m(1+\alpha/2d)}, 3^{-mq'/p}\right\} \leq \ep.
\end{equation*}
Also set $h:= 3^{m}\ep$ and $l := 3^{-m\alpha/2d}$. Note that $l\leq h\leq \ep^\gamma$ for some $\gamma(d,\Lambda,p)>0$. 
Applying Proposition~\ref{p.snapgrid} with $G=\overline F$ and $F^\ep(A,x):=F(A,\frac x\ep)$ in place of $F$, we obtain, for every fixed $E \geq Ch^\kappa$,
\begin{multline} \label{e.captcha}
\left\{ F\in \Omega\,:\, \sup_{x\in U}\left( u(x) - u^\ep(x,F) \right) \geq E \right\} \\ \subseteq \bigcup_{(A,y) \in \mathcal I(h) } \left\{ F\in \Omega\,:\, \mu(\ep^{-1}y+Q_{m},F_{A}-\overline F(A)) \geq c E^d\right\},
\end{multline}
where
\begin{multline*}\label{}
\mathcal I(h) := \Big\{ (A,y) \in \Sy \times B_1 \,:\, |A| \leq h^{\kappa-1}, \ \mbox{and both} \ l^{-1} A \ \mbox{and} \ h^{-1}y \\ \mbox{have integer entries} \Big\} .
\end{multline*}
We deduce that 
\begin{equation}
\label{e.captcha2}
 \sup_{x\in U}\left( u(x) - u^\ep(x,F) \right)_+^d  
 \leq C h^{\kappa d} + C \mathcal Y_m,
\end{equation}
where $\mathcal Y_m$ is the random variable
\begin{multline} \label{e.defYm}
\mathcal Y_m:= \sup\big\{ \mu\left(z+Q_{m},F_A - \overline F(A) \right) \,:\, z\in \Zd \cap B_{3^{m(1+\alpha/2d)}}, \\
 3^{m\alpha/2d}A \in \mathbb{S}^d\cap \Z^{d\times d} \cap B_{3^{m\alpha/d}}\big\}.
\end{multline}
Applying Corollary~\ref{c.mudeviation} to each $F_A$, in view of the definition of $\overline F(A)$ in~\eqref{e.Fbar}, we deduce that
\begin{equation*}\label{}
\P\left[ \mu(Q_{m},F_{A}-\overline F(A)) \geq (1+|A|)^d 3^{-m\alpha} t \right]\leq C\exp\left( -c3^{mq}t  \right).
\end{equation*}
A union bound (there are $C3^{m(d+\alpha/2)}\cdot 3^{m\alpha(d+1)/2}$ many elements in the supremum in~\eqref{e.defYm}), using also that $|A|^d\leq 3^{m\alpha/2}$ for every $A$ in the supremum in~\eqref{e.defYm}, then yields, for all $t\geq 1$,
\begin{equation*} \label{}
\P \left[  \mathcal Y_m \geq 3^{-m\alpha/2} t \right] \leq C 3^{m(d+\alpha/2)+m\alpha(d+1)/2} \exp\left(-c3^{mq}t \right).
\end{equation*}
Replacing $t$ with $1+t$, we deduce that, for every $t>0$,
\begin{equation*} \label{}
\P \left[  3^{m\alpha/2} \mathcal Y_m -1 \geq t \right] \leq C \exp\left(Cm-c3^{mq}(1+t) \right) \leq C \exp\left( -c3^{mq}t\right).
\end{equation*}
Now replace $t$ by $3^{-mq'}t$ to obtain, for every $t>0$,
\begin{equation*} \label{}
\P \left[ 3^{mq'}\left( 3^{m\alpha/2}\mathcal Y_m -1 \right)_+ \geq t \right] \leq C \exp\left(-c3^{m(q-q')} t  \right).
\end{equation*}
A union bound yields, for every $t\geq 1$,
\begin{equation} \label{e.crabcha}
\P \left[ \sup_{n\in\N}  3^{nq'}\left( 3^{n\alpha/2}\mathcal Y_n -1 \right)_+ \geq t \right] \leq C \sum_{n\in\N} \exp\left(-c3^{n(q-q')} t  \right) \leq C\exp\left( -ct \right).
\end{equation}
Define
\begin{equation*} \label{}
\mathcal X:= c\sup_{n\in\N}  3^{nq'} \left( 3^{n\alpha/2}\mathcal Y_n - 1 \right)_+
\end{equation*}
where $c>0$ is taken small enough that an integration of~\eqref{e.crabcha} yields 
\begin{equation*} \label{}
\E \left[ \exp\left( \mathcal X \right) \right] \leq C.
\end{equation*}
Returning to~\eqref{e.captcha2}, we get
\begin{equation*} \label{}
\sup_{x\in U}\left( u(x) - u^\ep(x,F) \right)_+^d  
 \leq C h^{\kappa d} +C \left(  3^{-mq'}\mathcal X +1 \right) 3^{-m\alpha/2}.
\end{equation*}
Using the definitions of~$h$ and $m$, we obtain, for some $\beta(p,d,\Lambda)>0$,
\begin{equation*} \label{}
\sup_{x\in U}\left( u(x) - u^\ep(x,F) \right)
 \leq C \left( 1+\mathcal X\ep^p \right) \ep^{\beta}.
\end{equation*}
Chebyshev's inequality now yields~\eqref{e.wtf}.
\end{proof}

\begin{remark}
The argument above gave a stronger result than the one stated in Theorem~\ref{t.full}. What we proved is that, for each $p\in(0,d)$, there exists $\beta(p,d,\Lambda) > 0$ and a nonnegative random variable $\mathcal X$ on $(\Omega,\F)$ satisfying
\begin{equation*} \label{}
\E \left[ \exp(\mathcal X) \right] \leq C(d,\Lambda,p,U)<\infty
\end{equation*}
and
\begin{equation} \label{e.aprioriform}
\sup_{x\in U} \left| u^\ep(x) - u(x) \right| \leq C \left( 1+\mathcal X\ep^p \right) \ep^{\beta} \left( K_0+\| g\|_{C^{0,1}(\partial U)} + \| f \|_{C^{0,1}(U)} \right). 
\end{equation}
This is stronger than Theorem~\ref{t.full} since the latter may be immediately recovered from~\eqref{e.aprioriform} and Chebyshev's inequality, but it also gives an error estimate independent of the data ($\mathcal X$ depends on the realization of the coefficients, but not, for example, on $\ep$, $g$ or $f$). 
\end{remark}

\section{Further remarks and some open problems}
\label{s.remarks}

We conclude with a discussion of generalizations and extensions of our results as well as some open problems. 
\subsection{Computing the effective coefficients}

One way of characterizing $\overline F(A)$ is to consider, for $\delta > 0$, the \emph{approximate cell problem}
\begin{equation*}\label{}
\delta w^\delta + F(A+D^2w^\delta,y) = 0 \quad \mbox{in} \ \Rd.
\end{equation*}
This has a unique stationary solution $w^\delta=w^\delta(\cdot,F,A) \in C^{0,1}(\Rd)$, which may be computed numerically using available (albeit slow) computational methods. The effective coefficients are given by the limit
\begin{equation*}\label{}
\lim_{\delta \to 0} \big| \delta w^\delta(0,F,A) + \overline F(A) \big| = 0. 
\end{equation*}
Using a comparison argument, Theorem~\ref{t.full} yields the following estimate for the previous limit, for $p<d$, $\alpha(p,d,\Lambda)$ as in the statement of the theorem and $C=C(d,\Lambda,K_0,|A|)$:
\begin{equation*}\label{}
\P \Big[ \big| \delta w^\delta(0,F,A) + \overline F(A) \big| > C\delta^\alpha \Big] \leq C \exp\left(-\delta^{-p} \right).
\end{equation*}
We leave the details to the reader. 

\subsection{What is the optimal exponent?} Theorem~\ref{t.full} is not the final word on the quantitative study of the stochastic homogenization of~\eqref{e.pde}. Now that an algebraic rate has been obtained, determining the best exponent~$\alpha$ in Theorem~\ref{t.full} is, in our opinion, the most important remaining task. This is beyond the reach of our current methods and, we expect, quite difficult.

In recent and striking papers, Gloria and Otto~\cite{GO1} and Gloria, Neukamm and Otto~\cite{GNO} proved \emph{optimal} error estimates for discrete elliptic equations in divergence form with i.i.d. coefficients, using a combination of regularity theory and concentration arguments. This suggests that it may be possible to develop an analogous theory for equations in nondivergence form, at least in the linear case. 

\smallskip

Short of finding the optimal $\alpha$ explicitly, it would still be interesting to further constrain it.  For example, can we replace the dependence of $\alpha$ on the ellipticity of $F$ with the ellipticity of $\overline F$?

\begin{question}
Can we show that the exponent $\alpha$ in Theorem~\ref{t.full} depends only on $d$ and $\overline \Lambda$, where $\overline\Lambda$ is the ellipticity of $\overline F$? If so, then in the linear case we would deduce that $\alpha$ depends only~$d$, as any constant-coefficient linear operator is, up to a change of variables, the Laplacian. 
\end{question} 

\subsection{Mixing conditions}
\label{mixing}

With small modifications, the arguments in this paper give appropriate quantitative error estimates under other hypotheses quantifying ergodicity. In this subsection, we explain the simple modifications needed to obtain results under a \emph{uniform mixing condition}, which is the most natural generalization of the finite range of dependence assumption. The arguments can also be modified to yield results under a \emph{strong} mixing condition (a weaker condition than uniform mixing); we leave the latter to the reader. For a discussion of mixing conditions, see~\cite[Chapter 17]{ILK}.

\begin{definition}
Let $\rho:(0,\infty) \to [0,\infty)$ be nonnegative, continuous and  decreasing with $\lim_{t\to \infty} \rho(t) = 0$. We say that a probability measure $\P$ on $(\Omega,\F)$ satisfies the \emph{uniform mixing condition with rate $\rho$} if, for every $U,V \subseteq \Rd$ and random variables $X$ and $Y$ such that $X$ is $\F(U)$--measurable and $Y$ is $\F(V)$--measurable, we have 
\begin{equation} \label{e.mixing}
\left| \cov\left[X;Y\right] \right| \leq \rho\left( \dist(U,V) \right) \var\left[X\right]^{1/2} \var \left[ Y \right]^{1/2}.
\end{equation}
\end{definition}

The arguments in this paper show that environments with a uniform mixing rate of $\rho$ have error estimates which are proportional to $\rho$, up to an algebraic rate of decay. We present the following analogue of Theorem~\ref{t.full} for environments satisfying a uniform mixing condition with an algebraic rate. The formulation of results for slower mixing rates (such as logarithmic rates) are left to the reader. 

\begin{theorem}
\label{t.mixing}
Suppose $\Lambda > 1$, $K_0>0$ and $\P$ is a probability measure on~$(\Omega(\Lambda),\F)$ satisfying~\eqref{e.Fubb} and \eqref{e.stat}. Suppose also that $\P$ satisfies the uniform mixing condition with rate $\rho(t) = At^{-\beta}$, for constants $A,\beta >0$. Then, if $0< \ep \leq 1$, $U \subseteq \R^d$ is a bounded smooth domain, $g \in C^{0,1}(\partial U)$, $f \in C^{0,1}(U)$, $u^\ep(\cdot, F) \in C(\bar U)$ denotes the unique solution of \eref{cswuep}, and $u \in C(\bar U)$ denotes the unique solution of \eref{cswu}, we have the estimate
\begin{equation*}
\P\bracks{ \sup_{x \in U} |u^\ep(x,F) - u(x)| \geq C \ep^\alpha} \leq C\ep^{\alpha},
\end{equation*}
where the exponent $\alpha > 0$ depends only on $d$, $\Lambda$ and $\beta$ and $C > 0$ depends only on $d$, $\Lambda$, $\beta$, $K_0$, $A$, $U$, $\| g \|_{C^{0,1}(\partial U)}$, and $\| f \|_{C^{0,1}(U)}$.
\end{theorem}

We continue with the modifications to the paper required to prove Theorem~\ref{t.mixing}. We assume without loss of generality that $\beta < d$.

\begin{itemize}

\item The only use of the finite range of dependence condition in the proof of Theorem~\ref{t.mudecay} is found in the proof of Lemma~\ref{l.concentrate}, precisely, in the bound~\eqref{e.usingfrd}. Rather than~\eqref{e.usingfrd}, the uniform mixing condition gives
\begin{equation*} \label{}
\qquad \left| \cov\left[\mu(Q_m(x),F) ;\mu(Q_m(y),F) \right] \right| \leq \rho\left( \dist(Q_m(x),Q_m(y)) \right) \var \left[ \mu(Q_,F)\right].
\end{equation*}
Using this bound in place of~\eqref{e.usingfrd}, we find, after a computation, that the right of~\eqref{e.varcov} is estimated from above by 
\begin{equation*} \label{}
\qquad C 3^{-\beta n} \var \left[ \mu(Q_n,F)\right].
\end{equation*}
This leads to the bound
\begin{equation*}
\qquad \E \bracks{\mu(Q_{m+n}(x),F)^2} \leq C (1 + \delta^{-1}) 3^{-\beta n} \var\bracks{\mu(Q_m,F)} + (1 + \delta) \E\bracks{\mu(Q_m,F)}^2
\end{equation*}
and we take  $\delta := 3^{-n\beta/2}$ to obtain the following result in place of~\eqref{e.concentrate}:
\begin{equation}
\label{e.concentrate-mix}
\qquad \E \bracks{ \mu(Q_{m+n},F)^2 } \leq  \E \bracks{ \mu(Q_m,F) }^2 + C 3^{-n\beta/2} \E \bracks{ \mu(Q_m,F)^2 }.
\end{equation}

\item The rest of the proof of Theorem~\ref{t.mudecay} proceeds essentially verbatim, and we obtain the statement of theorem for $\tau(d,\Lambda,\beta)\in (0,1)$ and $C(d,\Lambda,\beta,A)>0$.

\item The proof of Corollary~\ref{c.mudeviation} is a concentration argument that relies in an essential way on independence, so we cannot obtain an analogue of it.

\item To obtain Theorem~\ref{t.mixing}, we combine Proposition~\ref{p.snapgrid} with the extension of Theorem~\ref{t.mudecay} obtained above. This is similar to the proof of Theorem~\ref{t.full}, but requires slightly more care when selecting the parameter $m$, because we have only algebraic rather than exponential bounds for the probabilities. The necessary modifications are left to the reader.
\end{itemize}

\subsection{Further extensions and generalizations}

While Theorems~\ref{t.csw} and~\ref{t.full} are stated in terms of solutions to the Poisson-Dirichlet problem on bounded domains, deterministic comparison arguments give us analogous results for essentially any well-posed problem involving the operator $F$. The only issue is in adapting the proof of Proposition~6.2, which is straightforward. As such, we can obtain results for Neumann boundary conditions as well as time-dependent parabolic problems with appropriate initial conditions (e.g., the Cauchy problem) and/or boundary conditions (e.g., the Cauchy-Dirichlet problem). 

\smallskip

Similarly, the methods in this paper readily extend to the case of equations with lower-order terms, such as:
\begin{equation*}\label{}
F\left(D^2u,Du,u,x,\frac x\ep\right) = 0. 
\end{equation*}
Here we are thinking of equations with the ``usual" hypotheses, i.e., uniform ellipticity and Lipschitz continuity in each argument. Again, the only extra difficulty in this extension lies in obtaining a more general version of Proposition~\ref{p.snapgrid}. In other words, the most difficult part of the qualitative homogenization program, proving Theorem~\ref{t.mudecay}, goes through verbatim and the only remaining issue is in the deterministic link between Theorems~\ref{t.mudecay} and~\ref{t.full}.

We conclude with an open problem which is not as straightforward:

\begin{question}
Can the ideas in this paper be extended to the parabolic equations with time-dependent, random coefficients? The prototypical equation is
\begin{equation*}\label{}
u_t + F\left(D^2u,\frac x\ep, \frac t{\ep^2} \right) = 0,
\end{equation*}
where $F:\Sy\times \Rd \times \R \to \R$ and the underlying probability measure on equations is ergodic with respect to space-time shifts. What is the natural analogue of $\mu$?
\end{question}

\noindent{\bf Acknowledgements.}
S. Armstrong was partially supported by the Forschungsinstitut f\"ur Mathematik (FIM) of ETH Z\"urich during Spring 2013.  C. Smart was partially supported by NSF grant DMS-1004595. We thank Yves Capdeboscq and Xiaoqin Guo for pointing out some minor mistakes in the published version of this article that we have fixed in this revision.

\bibliographystyle{plain}
\bibliography{algebraic}

\end{document}